\documentclass[a4paper, 11pt]{amsart}   
\usepackage{mathptmx, amssymb,amscd,latexsym, eulervm}   
\usepackage{amsmath}
\usepackage{amsthm}
\usepackage{mathdots}
\usepackage{xcolor}
\usepackage{hyperref}
\hypersetup{
 colorlinks = true,
 linkcolor=blue,
 citecolor=orange,
 urlcolor=red
}

\usepackage{setspace}
\usepackage{tabularx}
\usepackage{comment}
\usepackage{amsfonts}
\usepackage{paralist}
\usepackage{aliascnt}
\usepackage[alphabetic]{amsrefs}
\usepackage{amscd}
\usepackage{blkarray}
\usepackage{booktabs}
\usepackage{enumitem}
\usepackage{setspace}
\usepackage[inner=2.4cm,outer=2.4cm, top = 3cm, bottom=2.7cm, marginparwidth=2cm]{geometry}
\usepackage{relsize}
\usepackage{multirow}
\usepackage{placeins}
\usepackage{float}

\usepackage[capitalize,nameinlink]{cleveref}

\usepackage{tikz, tikz-cd}
\usepackage{calligra,mathrsfs}
\usetikzlibrary{matrix}
\usetikzlibrary{arrows,calc,babel}
\usetikzlibrary{matrix,fit,decorations.pathreplacing}
\allowdisplaybreaks

\tikzset{curve/.style={settings={#1},to path={(\tikztostart)
    .. controls ($(\tikztostart)!\pv{pos}!(\tikztotarget)!\pv{height}!270:(\tikztotarget)$)
    and ($(\tikztostart)!1-\pv{pos}!(\tikztotarget)!\pv{height}!270:(\tikztotarget)$)
    .. (\tikztotarget)\tikztonodes}},
    settings/.code={\tikzset{quiver/.cd,#1}
        \def\pv##1{\pgfkeysvalueof{/tikz/quiver/##1}}},
    quiver/.cd,pos/.initial=0.35,height/.initial=0}

\newtheorem{thm}{Theorem}[section]
\newtheorem{lemma}[thm]{Lemma}
\newtheorem{cor}[thm]{Corollary}
\newtheorem{prop}[thm]{Proposition}

\theoremstyle{definition}
\newtheorem{example}[thm]{Example}
\newtheorem{notation}[thm]{Notation}
\newtheorem{setup}[thm]{Setup}
\newtheorem{remark}[thm]{Remark}
\newtheorem{definition}[thm]{Definition}

\newtheorem{conjecture}[thm]{Conjecture}
\newtheorem{question}[thm]{Question}

\numberwithin{equation}{section}

\newtheorem*{thm*}{Theorem}
\newtheorem*{lemma*}{Lemma}
\newtheorem*{cor*}{Corollary}
\newtheorem*{prop*}{Proposition}

\newtheorem{manualtheoreminner}{Theorem}
\newenvironment{manualtheorem}[1]{
  \renewcommand\themanualtheoreminner{#1}
  \manualtheoreminner
}{\endmanualtheoreminner}

\theoremstyle{definition}
\newtheorem*{definition*}{Definition}
\newtheorem*{example*}{Example}
\newtheorem*{remark*}{Remark}

\theoremstyle{definition}
\newtheorem{manualdefinitioninner}{Definition}
\newenvironment{manualdefinition}[1]{
  \renewcommand\themanualdefinitioninner{#1}
  \manualdefinitioninner
}{\endmanualdefinitioninner}

\theoremstyle{definition}
\newtheorem{manualdefinitionsinner}{Definitions}
\newenvironment{manualdefinitions}[1]{
  \renewcommand\themanualdefinitionsinner{#1}
  \manualdefinitionsinner
}{\endmanualdefinitionsinner}

\theoremstyle{plain}

\theoremstyle{remark}
\newtheorem*{acknowledgements}{Acknowledgements}

\newcommand{\Spec}{\mathrm{Spec}}
\newcommand{\Proj}{\mathrm{Proj}}
\newcommand{\Hilb}{\mathrm{Hilb}}

\newcommand{\rank}{\mathrm{rank}}
\newcommand{\codim}{\mathrm{codim}}

\newcommand{\Sym}{\mathrm{Sym}}

\newcommand{\Z}{\mathbb{Z}}

\newcommand{\kk}{{\Bbbk}}

\newcommand{\bfA}{\mathbf{A}}

\newcommand{\bfP}{\mathbf{P}}

\newcommand{\bfF}{\mathbf{F}}

\newcommand{\bfw}{\mathbf{w}}
\newcommand{\bfu}{\mathbf{u}}

\newcommand{\mfm}{{\mathfrak{m}}}

\newcommand{\mcR}{{\mathcal{R}}}

\newcommand{\mcC}{{\mathcal{C}}}
\newcommand{\mcX}{{\mathcal{X}}}

\newcommand{\mcO}{{\mathcal{O}}}

\newcommand{\mcP}{{\mathcal{P}}}

\newcommand{\mcL}{{\mathcal{L}}}

\newcommand{\floor}[1]{\lfloor #1 \rfloor}

\title{Minimal degree}

\newcommand{\ritvik}[1]{{\color{cyan} \sf $\clubsuit$ Ritvik: [#1]}}

\makeatletter
\@namedef{subjclassname@2020}{
  \textup{2020} Mathematics Subject Classification}
\makeatother

\begin{document}

\author[M. \,Banks, R.\,Ramkumar]{Maya~Banks and Ritvik~Ramkumar}
\address{(Maya Banks) Department of Mathematics, Statistics, and Computer Science \\ University of Illinois Chicago \\ USA}
\email{mayadb@uic.edu}
\address{(Ritvik Ramkumar) Department of Mathematics\\University of Notre Dame\\South Bend \\USA}
\email{rramkuma@nd.edu}

\title{Varieties of minimal degree in weighted projective space}


\maketitle

\begin{abstract} 
We initiate a study of varieties of minimal degree in weighted projective spaces. We call a weighted projective space $\mathbf{P}(w_0,\dots,w_n)$ divisible if $w_i \mid w_{i+1}$ for all $i$. We provide sharp bounds for when a non-degenerate subvariety of a divisible weighted projective space has minimal degree. We define a weighted notion of $1$-generic matrices and, in analogy with the classical theory, show that there is a theory of weighted determinantal scrolls. Moreover, we characterize precisely when these have minimal degree and determine their weighted $N_p$ properties, and tie this to two weighted notions of regularity. Finally, we propose conjectural bounds for more general weighted threefolds and pose several natural questions. Throughout, we highlight the differences between this theory and the classical case.
\end{abstract}

\section{Introduction}
A subvariety $X \subseteq \bfP^n$ of codimension $c$ that is not contained in a hyperplane satisfies
$\deg(X) \geq 1 + c$.
If equality is attained, then $X$ is called a \textbf{variety of minimal degree}. The celebrated results of del Pezzo and Bertini \cite{Bert23, DelPezzo, Eisenbud_Harris_minimal} provide a complete classification of these varieties: $X$ is either a quadric hypersurface, a rational normal scroll, or a cone over the Veronese surface $\bfP^2 \hookrightarrow \bfP^5$. Eisenbud and Goto \cite{EG84} showed that varieties of minimal degree are completely characterized by their homological properties; namely, $X$ is of minimal degree if and only if $I_X$ is $2$-regular. This syzygetic perspective culminated in the theory of small schemes, vastly generalizing the results of del Pezzo and Bertini \cite{EGHP06}.

These results fit into a larger program devoted to understanding the geometry of syzygies in projective space, beginning with Green's theorem on linear syzygies of smooth curves \cite{Green84a, Green84b}. A recent trend has focused on extending the ``geometry of syzygies'' to a multigraded setting, exchanging embeddings into $\bfP^n$ for embeddings into more general toric varieties, including weighted projective spaces. For some examples, see \cite{MacSmith04, HSS06, BES20, bruce2022bounds, Cobb2023Syzygies, Hanlon_Hicks_Lazarev_2024, Brown_Erman_2024, berkesch2025cellular, Cranton_Heller_2025}.
These extensions have seen applications beyond the realm of toric varieties; for example, the weighted Castelnuovo--Mumford regularity introduced in \cite{Ben04} has been applied to the study of group cohomology and invariant rings \cite{Symonds2010,Symonds2011}.

Of particular relevance to us are \cite{Brown_Erman_Positivity, Brown_Erman_Linear}, in which the authors develop analogues of Green’s $N_p$ conditions for weighted projective spaces $\bfP(w_0,\dots,w_n)$, discuss several notions of Castelnuovo-Mumford regularity in the weighted setting, and prove an analogue of Green's theorem for curves in $\bfP(1,\ldots, 1, m, \ldots, m)$. These results show that the interplay between geometry and syzygies remains quite robust in the weighted setting. We further explore these connections by initiating a study of minimal degree varieties in weighted projective spaces.
A theory of varieties of minimal degree in $\bfP(\bfw)$, beyond being interesting in its own right, could pave the way for a tighter connection between the geometry of subvarieties in $\bfP(\bfw)$ and their free resolutions. For instance, as suggested in \cite[Section 7.2]{Brown_Erman_Linear}, such a theory could lead toward a weighted analogue of the Green–Lazarsfeld gonality conjecture, relating the syzygies of a curve to its special embeddings into varieties of minimal degree.

The primary goal of this paper is to initiate a study of varieties of minimal degree in weighted projective space, including the development of a notion of weighted scrolls, and to explore the relationship between degree and syzygies of weighted projective varieties. To this end, we consider the following questions:
\begin{enumerate}[label=(Question \arabic*), leftmargin=*]
\item What is the minimal degree of a non-degenerate subvariety of a weighted projective space?
\item Which subvarieties attain this minimal degree?
\item Do minimal degree varieties in weighted projective space have simple syzygies?
\end{enumerate}

In answering these questions, we mainly focus on a class of weighted projective spaces where the weights satisfy a natural divisibility condition (though all preliminary results of \cref{section_preliminaries} hold in general).

\begin{definition}[see \Cref{setup}] \label[definition]{divisible_wps} A weighted projective space $\bfP(\bfw)$ is \textbf{divisible} if $w_i | w_{i+1}$ for each $i$. 
We can group the weights together and write $\bfP(\bfw) =  \bfP(m_0^{a_0}, m_1^{a_1}, \ldots, m_k^{a_k})$ with 
\begin{itemize}
\item $m_0 = 1$ and $a_0 \geq 2$,
\item $m_i < m_{i+1}$ and $m_i|m_{i+1}$.
\end{itemize}
\end{definition}

The reason for restricting to divisible weighted projective spaces is two-fold. From a practical standpoint, behavior of weighted projective spaces and their subschemes tends to depend heavily on what can at times be quite subtle numerics of the weights; making the problem tractable often requires some constraints on the divisibility relations that show up. Beyond this, we have mentioned that we are particularly motivated by \cite{Brown_Erman_Linear}, wherein the authors study weighted $N_p$ conditions for curves in projective spaces of the form $\bfP(1^{a_0}, m^{a_1})$. Divisible weighted projective spaces are a natural generalization of this setting. In \cref{section_threefold}, we leave the divisible setting and explore some of the phenomena that arise more generally.

In the case of divisible weighted projective spaces, we provide the following answer to Question 1.

\begin{manualtheorem}{A}[\Cref{prop_degree_bound} and \Cref{cone_degree}] \emph{
Let $\bfP(1^{a_0}, m_1^{a_1}, \dots, m_k^{a_k})$ be a divisible weighted projective space and  $X$ a non-degenerate subvariety of dimension $d \leq a_k$. Then 
\begin{align*}
\deg(X) \geq  \frac{1}{m_k^{d-1}}\left(a_0 -1+ \frac{a_1}{m_1} + \cdots + \frac{a_k}{m_k} + \frac{1-d}{m_k} \right)
\end{align*}
and this bound is sharp. Any variety attaining this bound is rational.}
\end{manualtheorem}
In fact, in \cref{prop_degree_bound} we prove a related bound without the requirement that $d\leq a_k$, at the expense of a definitive rationality statement.

The answer to Question 2 is more complicated. Varieties of minimal degree need not be arithmetically Cohen--Macaulay. Furthermore, there are in general infinitely many varieties of minimal degree up to projective equivalence. These pathologies are already present for curves in simple weighted projective spaces; $\bfP(1,1,2,2)$ contains a curve of minimal degree that is not arithmetically Cohen--Macaulay, and $\bfP(1^2,2^4)$ contains infinitely many projective equivalence classes of minimal degree curves (\Cref{three_curves_1} and \Cref{dimension_count}).

In projective space, every variety of minimal degree is cut out by the $2\times 2$ minors of a matrix of linear forms. In light of this and the challenges to the classification problem discussed above, our approach to Question (2) is to understand analogues of linear determinantal varieties for weighted projective spaces. To this end, we introduce what we believe to be a robust analogue of rational normal scrolls in the weighted setting. Our theory of weighted scrolls is rooted in the determinantal presentation in the classical case: we define a class of weighted projective varieties cut out by the minors of matrices satisfying a weighted version of the classical $1$-genericity condition.

\begin{manualdefinitions}{\ref{def_gen_entries} and \ref{def_1generic}}
Let $M$ be a matrix with homogeneous entries in $S(\mathbf{w}) = \kk[x_0,\ldots, x_n]$ where $w_i = \deg(x_i)$. Assume that all entries in a given column have the same degree. The \textbf{profile} of $M$ is the sequence of degrees of the entries in its first row. A \textbf{generalized row} of a matrix $M$ is a nontrivial $\kk$-linear combination of the rows of $M$, and a \textbf{generalized entry} is a nontrivial $\kk$-linear combination of the entries of a generalized row. We say that the matrix $M$ is \textbf{pseudo 1-generic} if the entries of any generalized row are a regular sequence in $S(\bfw)_{(x_0,\ldots, x_n)}$.  If $M$ is pseudo 1-generic and its profile can be obtained by deleting entries from $(w_0, \ldots, w_n)$, then we say $M$ is \textbf{1-generic}.
\end{manualdefinitions}

When $S$ is standard graded, pseudo 1-genericity is equivalent to the usual 1-genericity \cite[Section 6B]{eisenbud:syzygies}. In this case, all the entries of $M$ must be linear and pseudo 1-genericity is equivalent to the nonvanishing of any generalized entry. Just as in the classical case, the vanishing locus of the maximal minors of a $p\times q$ pseudo 1-generic matrix over $S$ is an arithmetically Cohen--Macaulay scheme of codimension $q - p + 1$ (\Cref{thm_CM}). However, we cannot ensure integrality in general (\Cref{surface_possibilities}).

\begin{manualdefinition}{\ref{def_weightedScroll}} \label{def:intro_scroll}
Let $M$ be a 1-generic matrix in $S(\mathbf{w})$ with two rows. 
If $X = V(I_2(M)) \subseteq \mathbf{P}(\mathbf{w})$ is integral we call $X$ 
a \textbf{(weighted) determinantal scroll}.  
\end{manualdefinition}

This definition is not restrictive; we show in \Cref{prime_is_generic} that if $X$ is integral and defined by the minors of a matrix $M$ with two rows, then $M$ must be pseudo 1-generic. Furthermore, for divisible spaces, one can always find 1-generic matrices. If an entry of the profile increases (for instance to be larger than a weight in the ambient weighted projective space), then the degree of the scroll increases as well; thus, if we are interested in varieties of minimal degree, 1-genericity is a natural condition.

In another departure from the classical case, we find that not all weighted determinantal scrolls have minimal degree. We classify those that do have minimal degree in terms of their profiles, and appeal to the Kronecker--Weierstrass characterization of matrix pencils to provide a structure theorem. 
For simplicity, we again restrict to $d \leq a_k$, the multiplicity of the largest weight.

\begin{manualtheorem}{B}[\Cref{thm_minDegStructure} and \Cref{cone_degree}] \emph{
Let $X$ be a $d$-dimensional determinantal scroll in a divisible weighted projective space and $M$ its associated $1$-generic matrix. 
\begin{enumerate} 
\item $X$ has minimal degree if and only if its profile is $$(1^{a_0-1},m_1^{a_1},\dots,m_{k-1}^{a_{k-1}},m_k^{a_k+1-d}).$$ 
\item Let $M_i$ denote the submatrix of $M$ whose entries have degree $m_i$. If $X$ has minimal degree, then, up to a change of coordinates,
\begin{itemize}
\item $M_0$ is a weighted scroll block (see \Cref{matrix_decomposition_notation}),
\item for  $1 \leq i \leq k-1$, $M_i$ is made up of Jordan blocks and at most one nilpotent block of size one,
\item $M_k$ is made up of Jordan blocks, scroll blocks, and at most one nilpotent block of size one.
\end{itemize}
\end{enumerate}}
\end{manualtheorem}
We actually prove something stronger in \Cref{thm_minDegStructure}; we show that if $X \subseteq \bfP(\bfw)$ is defined by a $1$-generic matrix with the above profile, then $X$ is integral. In particular, this is automatic when $\dim(X) =1$.
Restricting to dimension one, we get a more subtle analogue of a rational normal curve. Consider the weighted projective stack $\mcP(\bfw) := [(\bfA(\bfw)\setminus 0)/\mathbf{G}_m]$ with coarse space morphism $\pi: \mcP(\bfw) \to \bfP(\bfw)$.

\begin{manualtheorem}{C}[\Cref{cor_smooth_curves} and \Cref{stacky_parameterization}] \emph{
Let $\bfP(\bfw)$ be a divisible weighted projective space. Let $C $ be a curve cut out by the minors of a 1-generic matrix with two rows and let $\mathcal{C} = \pi^{-1}(C)$. 
 Then
\begin{enumerate}
\item $C$ is a curve of minimal degree,
\item $C$ is isomorphic to $\bfP^1$,
\item  $\mathcal{C}$ may be singular, and we get a ``global parameterization" of $\mathcal{C}$ and thus of $C$, via
$$
\mathcal{C}_{\mathrm{sm}} \to \mathcal{C} \to C,
$$
where $\mathcal{C}_{\mathrm{sm}}$ is a root stack over $\bfP^1$, corresponding to a desingularization of $\mathcal{C}$.
\end{enumerate}}
\end{manualtheorem}

In particular, while $C$ is isomorphic to $\bfP^1$, it is not always possible to write down a globally defined map from $\bfP^1$ to $\bfP(\bfw)$ whose image is $C$. By passing to the stack, we circumvent this issue and explicitly construct the map. However, this introduces a new issue: in general $\mathcal{C}$ may not be smooth, in which case our ``global parameterization'' is, at best, an injective map.

\begin{example} 
Consider the curves $C_1, C_2\subset \bfP(1,1,3,3,3)$ defined by the maximal minors of the $1$-generic matrices
\[
M_1 = \begin{pmatrix}
x_1 & y_1    & y_2 & y_3\\
x_2 & x_1^3 & y_1 & y_2
\end{pmatrix}
\text{ and } \,
M_2 = \begin{pmatrix}
x_1 & y_1 & y_2 & y_3\\
x_2 & x_1^3 &y_2 + x_1^3& -y_3+x_1^3
\end{pmatrix}.
\]
We can parameterize $C_1$ via the closed immersion $\bfP^1\to \bfP(1,1,3,3,3)$ given by
\begin{align*}
[s : t]&\mapsto [st : t^2 : s^4t^2: s^5t : s^6].
\end{align*}
On the other hand, we parameterize $C_2$ by 
\[
\mcC_\text{sm} = \sqrt[3]{(\bfP^1, [0:1])}\times_{\bfP^1}\sqrt[3]{(\bfP^1, [1:1])}\times_{\bfP^1}\sqrt[3]{\bfP^1, [1:-1]} \to \mathcal{P}(1,1,3,3,3)
\]
given by
\[
[s:t] \mapsto [suvw : tuvw : s^4(t-s)(t+s) : s^4t(t+s) : s^4t(t-s)]
\]
where $u^3 = t$, $v^3=(t-s)$, and $w^3 = (t+s)$. More precisely, $u, v,$  and $w$ are sections of the root line bundles $\sqrt[3]{\mcO_{\bfP}([0:1])}$, $\sqrt[3]{\mcO_{\bfP}([1:1])},$ and $\sqrt[3]{\mcO_{\bfP}([1:-1])}$ respectively on the stacky curve $\mcC_\text{sm}$. 
\end{example}

\begin{remark} Consider $\bfP(1^d,m^m)$, and let $M$ be a $1$-generic matrix of profile $(1^{d-1},m^m)$ having a single Jordan block with $\epsilon = 0$ in degree $m$. It is shown in \cite[Theorem A]{Davis_Sobieska} that $I_2(M)$ cuts out a smooth rational curve $C$, referred to as a type $(d, m)$ rational normal curve. This curve is defined as the image of a closed immersion $\bfP^1 \to C$ via a weighted linear series.
\end{remark}

We turn finally to Question 3 for determinantal weighted scrolls. Classically, the simplicity of syzygies is measured either via the Castelnuovo-Mumford regularity or the $N_p$ properties. In weighted projective space, there are several distinct notions of regularity, of which we focus on two: \textbf{weighted regularity} and \textbf{Koszul regularity} (see \Cref{section_regularity}). 
We explicitly compute the Betti numbers of the determinantal scrolls and analyze their weighted and Koszul regularity as well as their weighted $N_p$ properties (\Cref{def_wNk}). 

\begin{manualtheorem}{D}[\Cref{minDeg_minReg} and \Cref{cor_scrolls_wNk}] \emph{
Let $X$ be a $d$-dimensional determinantal scroll in a divisible weighted projective space. 
\begin{enumerate}
\item $X$ is a variety of minimal degree if and only if its weighted (resp. Koszul) regularity is minimal among all the $d$-dimensional determinantal scrolls.
\item $X$ is weighted $N_p$ for all $p \geq 0$.
\end{enumerate}}
\end{manualtheorem}

This result provides a notable juxtaposition to projective space, where the linearity of the resolution exactly captures whether a variety has minimal degree. In contrast, there is no reasonable notion of a “linear resolution” that detects minimal degree in the weighted setting. In particular, we can construct varieties whose resolutions are “Koszul linear” in the sense of \cite{Brown_Erman_Linear}, yet whose degree is much larger than the minimal possible. On the other hand, we find a closer relationship between the degree of a weighted projective variety and its regularity: for weighted determinantal scrolls, having minimal degree is equivalent to having minimal regularity (both Koszul and weighted). When the degree is not minimized or a variety is not Cohen--Macaulay, the nature of this relationship appears subtle; we explore this further in \cref{section_regularity} and the examples therein.

In \cref{section_threefold}, we explore non-divisible weighted projective spaces, leveraging our results for divisible spaces to obtain degree bounds even in the absence of divisibility. We work this out for curves in weighted projective threefolds $\bfP(1,1,m,n)$ for arbitrary $m,n$. We also study determinantal curves in a general $\bfP(w_0,w_1,w_2,w_3)$ and give an example of a weighted projective threefold in which a curve of minimal degree cannot be determinantal. Finally, we conclude with a number of natural follow-up questions.

\begin{acknowledgements}
We thank Izzet Coskun, David Eisenbud, Daniel Erman, and Andres Fernandez Herrero for helpful conversations. Ritvik Ramkumar was partially supported by NSF grant DMS-2401462. Maya Banks was partially supported by NSF RTG grant DMS-2037569.
\end{acknowledgements}

\section{Quasi-polynomials and degree of subschemes}\label{section_preliminaries}
We begin this section with some notation that will be fixed throughout the paper. 
We then introduce some preliminary definitions and results concerning Hilbert quasi-polynomials of subschemes of weighted projective space and describe how to use them to compute degree.

\begin{setup} \label[setup]{setup} The field $\kk$ is always assumed to be algebraically closed. We use $S$ to denote the standard graded polynomial ring $\kk[x_0,\dots,x_n]$. Given $\bfw =\{w_0, \dots, w_n\} \subseteq \mathbf{N}^{n+1}$ a sequence of weakly increasing positive integers, denote by $S(\bfw)$ the polynomial ring $\kk[x_0,\dots,x_n]$ graded by $\deg(x_i) = w_i$. We usually group common weights together and write $\bfw = (m_0^{a_0},m_1^{a_1},\dots,m_k^{a_k})$ with $m_i < m_{i+1}$. The subspace of polynomials in $S(\bfw)$ of degree $i$ is denoted by $S(\bfw)_i$. There are two instances, both clear from context, where we will use a different notation for the variables:
\begin{itemize}
\item If there are at most three distinct degrees, we use the variables $x_i,y_i,z_i$ and begin indexing the variables at $1$. For example, $S(1,1,2,2,2,3) = \kk[x_1,x_2,y_1,y_2,y_3,z_1]$ with $\deg(x_i) =1$, $\deg(y_i) = 2$ and $\deg(z_i) = 3$.
\item If $\bfw = (m_0^{a_0},m_1^{a_1},\dots,m_k^{a_k})$, let 
$S(\bfw) = \kk[x_{0,1}, \ldots, x_{0,a_0}, \ldots, x_{k,1},\dots, x_{k,a_k}]$
with $\deg(x_{i,j}) = m_i$.
\end{itemize}

We define $\mathbf{A}(\bfw) = \Spec(S(\bfw))$ and  $\bfP(\bfw) = \Proj(S(\bfw))$. The latter is \textbf{weighted projective space} and can alternatively be defined as the quotient $(\mathbf{A}^{n+1}\setminus0) /\mathbf{G}_m$ where $\mathbf{G}_m$ acts on $\mathbf{A}^{n+1}$ by $x \mapsto (t^{w_0},t^{w_1},\dots, t^{w_n})\cdot x$. The \textbf{weighted projective stack} is the quotient stack $\mathcal{P}(\bfw) = [(\mathbf{A}^{n+1}\setminus0) /\mathbf{G}_m]$.

A weighted projective space $\bfP(\bfw)$ is \textbf{well-formed} if the action of $\mathbf{G}_m$ has trivial stabilizers in codimension $1$, equivalently, $\gcd(w_0,\dots,\widehat{w_i},\dots,w_{n})=1$ for all $i$. Every weighted projective space is isomorphic to a well-formed one. All our weighted projective spaces are assumed to be well-formed.

For a subscheme $X \subseteq \bfP(\bfw)$, $I_X \subseteq S(\bfw)$ denotes its defining ideal and $HF_{X}(t)$ denotes the Hilbert function of its coordinate ring $S(\bfw)/I_X$. A subscheme $X \subseteq \bfP(\bfw)$ is said to be \textbf{arithmetically Cohen--Macaulay} if the coordinate ring $S(\bfw)/I_X$ is Cohen--Macaulay. A \textbf{subvariety} of $\bfP(\bfw)$ is a closed integral subscheme. For more on the theory of weighted projective spaces, see \cite{Dolgachev82, RB86}. 

\end{setup}

\begin{definition}
A closed subscheme $X \subseteq \bfP(\bfw)$ is \textbf{non-degenerate} (respectively \textbf{degenerate}) if the defining ideal $I_X$ is (resp. is not) contained in $(x_0,\dots,x_n)^2$.
\end{definition}

\begin{definition}
A \textbf{quasi-polynomial of degree $d$} is a function $Q: \mathbb{Z}\to \kk$ given by 
\[
Q(t) = c_d(t)t^d + c_{d-1}(t)t^{d-1} + \cdots + c_1(t)t + c_0(t)
\]
where $c_i: \mathbb{Z}\to \kk$ is periodic and $c_d\neq 0$. The \textbf{period} of $Q$ is the smallest positive integer $p$ such that $c_i(t) = c_i(t+mp)$ for all $i$ and all $t,m\in\mathbb{Z}$. For each $0 \leq j \leq p-1$, the \textbf{$j$-th strand} of $Q$ is the polynomial
$$
Q^j(t) = c_d(j)t^d + \cdots + c_1(j)t + c_0(j).
$$
The \textbf{$j^{th}$ leading term} is $LT_j = c_d(j)t^d$ and the \textbf{$j^{th}$ leading coefficient} is $LC_j = c_d(j)$.
\end{definition}

\begin{prop}[{\cite[Theorems 1 and 2]{Bruns_Ichim}}] \label[prop]{quasi_poly} If $X \subseteq \bfP(\mathbf{w})$ is a closed subscheme of dimension $d$, there is a quasi-polynomial $HQ_X(t)$ of degree $d$.
such that $HQ_X(t) = HF_X(t)$ for all $t \gg 0$. Moreover, 
\begin{enumerate}
\item the period of $HQ_X$ divides $\text{lcm}(w_0, \ldots, w_n)$, and
\item if $w_0 =1$ and $X$ is non-degenerate, then the leading coefficient of $HQ_X$ is constant.
\end{enumerate}
The polynomial $HQ_X$ is called the \textbf{Hilbert quasi-polynomial} of $X$. 
\end{prop}

\begin{definition} Let $X\subseteq \bfP(\mathbf{w})$ be a $d$-dimensional closed subscheme. The \textbf{degree of $X$} is 
$$
\deg(X) = (d!) LC_0(HQ_X).
$$
\end{definition}

\begin{remark} The leading coefficient of the Hilbert quasi-polynomial may not be constant in general. Consider a ring concentrated only in even degrees; for example, the coordinate ring of $X = V(x_1) \subseteq \bfP(1,2,2)$. In this case, the period of $HQ_X$ is two, and $Q^0(t) = \frac{1}{2}t+1$ and $Q^1(t) = 0$. 
By \cref{quasi_poly}, all subschemes that we will consider have Hilbert quasi-polynomials with constant leading coefficients.
\end{remark}

\begin{remark} We can also compute the degree by using an embedding into projective space.
Let $q = \text{lcm}(w_0,\dots,w_n)$ and choose $r >0$ so that $\mcO_{\bfP(\bfw)}(rq)$ is very ample. Let $\nu: \bfP(\bfw) \to \bfP^N$ be any embedding by a complete linear series $\mcO_{\bfP(\bfw)}(rq)$ (i.e. a Veronese embedding of $\bfP(\bfw)$).
Then we have 
$$
\deg(X) = \frac{\deg(v(X))}{(rq)^{d}},
$$
where the right hand side is the usual degree in projective space. 
This follows from the fact that $HF_{v(X)}(t) = HF_X((rq)t)$, and thus $HP_{v(X)}(t) = HQ_X((rq)t)$.
\end{remark}

Some embeddings in weighted projective space are given by sections of powers of a line bundle. For subschemes obtained thusly, we can compute the degree using the line bundle and sections in question.

\begin{definition}\label{def_weighted_series}
Let $X$ be a projective scheme and $\mcL$ a line bundle on $X$. We denote its $\mathbb{Z}$-graded section ring by  $\mcR_\mcL = \bigoplus_{i\geq 0}H^0(X,\mcL^i)$. A \textbf{weighted series} is a $\mathbb{Z}$-graded $\kk$-vector space $W \subseteq \mcR_\mcL$. The choice of a basis $s_0,\dots,s_n$ of $W$, where $s_i \in W_{w_i} \subseteq H^0(X,\mcL^{w_i})$, induces a rational map $\varphi_W:X \dashrightarrow \bfP(\bfw)$. If the intersection of the $V(s_i)$ is empty, then $W$ is said to be \textbf{basepoint-free} and the rational map $\varphi_W$ is a morphism.
\end{definition}

\begin{lemma}\label[lemma]{lemma_line_bundle_degree}
Let $X$ be a $d$-dimensional projective variety, $\mcL$ an ample line bundle on $X$ and $W$ a basepoint-free weighted series. Let $q$ be a multiple of $\text{lcm}(m_i : \dim_\kk(W_i)>0)$ such that $\mcL^{\otimes q}$ is very ample. Then the degree of $\varphi_W(X)$ is
\[
\lim_{i\to\infty}\frac{\deg(\mcL) \cdot\dim_\kk(\mcR_W)_{(iq)}}{\dim_\kk(\mcR_\mcL)_{(iq)}}
\]
\end{lemma}

\begin{proof}
Note that $\dim_\kk(\mcR_W)_j$ is $HF_{\varphi_W(X)}(j)$ which is eventually equal to the Hilbert quasi-polynomial of the image of $X$. Similarly, $\dim_\kk(\mcR_\mcL)_j$ is eventually equal to the Hilbert polynomial of $\mcL$. Taking the limit we obtain
$$
\frac{\deg\mcL\cdot q^dLC_0(HQ_{\varphi_W(X)})}{q^dLC(HP_\mcL)}
= \frac{\deg\mcL\cdot q^d\frac{\deg \varphi_W(X)}{d!}}{q^d\frac{\deg\mcL}{d!}}
= \deg \varphi_W(X). \qedhere
$$
\end{proof}

The degree of a weighted projective variety exhibits behaviors not seen in the standard setting. Firstly, as the preceding lemma suggests, the degree of a variety (even a smooth one) need not be an integer; this is expected, however, given that weighted projective spaces are inherently singular. Secondly, taking a cone over a variety can alter its degree. Consequently, the same ideal viewed in different graded rings can define varieties with different degrees, as the next example illustrates.

\begin{example}
Consider the rational curve in $C \subseteq \bfP(1,1,2,2)$ defined by the weighted series $\{s^2, st, st^3, t^4\}\subset H^0(\bfP^1, \mcL)\oplus H^0(\bfP^1, \mcL^2)$ with $\mcL = \mcO(2)$. Using \cref{lemma_line_bundle_degree} we see that $C$ has degree 2.
The cone over $C$ in $\bfP(1,1,1,2,2)$ is a degree 2 surface while the cone over $C$ in $\bfP(1,1,2,2,2)$ is a degree $1$ surface.
\end{example}

\begin{definition}
Let $X\subset \bfP(\bfw)$ be a closed subscheme and let $\bfw'$ be a weight sequence obtained by appending a positive integer $m$ to $\bfw$. The \textbf{$m$-cone over $X$} in $\bfP(\bfw')$ is defined by $I_XS(\bfw')$.
\end{definition}
\begin{lemma}\label[lemma]{lemma_cone_degree}

Let $X\subset\bfP(\bfw)$ be a $d$-dimensional closed subscheme with Hilbert quasi-polynomial $HQ_X(t)$. 
If $X'$ is an $m$-cone over $X$, then 
\[
\deg(X') = \frac{d!}{q'}\sum_{\substack{0\leq j< q\\ \emph{\text{gcd}}(m,q)|j}} c_d(j)
\]
where $q = lcm(w_0, \ldots, w_n)$ and $q' = lcm(w_0, \ldots, w_n, m)$.
\end{lemma}

\begin{proof}
We compute $\deg(X')$ from the Hilbert quasi-polynomial. 
For $t\gg 0$, the Hilbert functions of $X$ and $X'$ are related by
\begin{equation}\label{eq_cone_hqp}
HF_{X'}(t) = HF_X(t) + HF_X(t-m) + HF_X(t - 2m) + \cdots + HF_X(t - \lfloor t/m \rfloor m)
\end{equation}
The same relation holds for the Hilbert quasi-polynomials $HQ_X$ and $HQ_{X'}$. To compute the degree, we need only the leading coefficient of the $0$-th strand of $HQ_{X'}$, so we assume that $t$ is divisible by $q'$, which in turn implies that $t$ is divisible by $q $.

Now we use the method of finite differences to compute the leading coefficient of $HQ^0_{X'}(t)$. 
For a polynomial $P$, recursively define $\Delta^n_{h}P(t) = \Delta^{n-1}_h P(t) - \Delta^{n-1}_h P(t-h)$ with $\Delta^1_{h}P(t) = P(t) - P(t-h)$.
For a degree $n$ polynomial $P$ with leading coefficient $c_n$, we have $\Delta^n_h P(t) = c_n h^n n!$, which means that 
\begin{equation}\label{eq_findif_deg}
\deg(X') = \frac{\Delta^{d+1}_{q'}HQ^0_{X'}(t)}{(q')^{d+1}}.
\end{equation}
From \Cref{eq_cone_hqp}, we have 
\[
\Delta_{q'}^{d+1}HQ^0_{X'}(t) 
= \Delta_{q'}^d\Delta_{q'}^1HQ^0_{X'}(t)
= \Delta_{q'}^d\left(\sum_{j=0}^{q'/m-1}\Delta_{m}^1HQ^0_{X'}(t-jm)\right)
= \Delta_{q'}^d \left( \sum_{i=0}^{q'/m-1}HQ^0_X(t-im)\right).
\]
Because the period of $HQ_X(t)$ divides $q'$, we can rewrite this further as 
\[
\sum_{i=0}^{q'/m-1}\Delta^d_{q'}HQ_X(t-im) 
= \sum_{i=0}^{q'/m-1}(q')^d d!\ c_d(t-im)
= (q')^d d!\sum_{i=0}^{q'/m-1} c_d(im)
= (q')^d d! \sum_{\substack{0\leq j< q\\ \text{gcd}(m,q)|j}} c_d(j).
\]
The second equality follows from the fact that the period of $HQ_X$ divides $q'$, which in turn divides $t$. The third equality follows by counting, for each $j$ between $0$ and $q-1$, the number of $i$ for which $im \equiv j \bmod q$. Combining with \Cref{eq_findif_deg} yields the result.
\end{proof}

\begin{remark} \label[remark]{cone_rem}
When the leading coefficient of $HQ_X$ is constant (see \Cref{quasi_poly}), \Cref{lemma_cone_degree} relates the degree of $X'$ to the degree of $X$ in a natural way:
\[
\deg(X') = \left(\frac{1}{q'}\right)\left(\frac{q}{\gcd(m,q)}\right)\deg(X) = \left(\frac{1}{q'}\right)\left(\frac{\text{lcm}(m,q)}{m}\right)\deg(X) = \frac{1}{m}\deg(X).
\]
\end{remark}

Much like in the standard-graded case, we can read the degree from the Hilbert series by first reducing it to a specific form.

\begin{prop}\label[prop]{thm_degree} 
Let $X \subseteq \bfP(\bfw)$ be a closed subscheme of dimension $d$ such that the reduced Hilbert series is of the form  $\frac{P_X(t)}{\prod_{i=0}^{d}(1-t^{w'_i})}$  where $P_X(t)$ is a rational function with no pole at $t=1$ and $(w'_0,w'_1, \dots, w'_{d})$ is a subsequence of $\bfw$. 
Then 
$$
\deg(X) =   \frac{P_X(1)}{w'_{0}w'_1\cdots w'_{d}}.
$$
\end{prop}
\begin{proof} Let $d$ be the dimension of $X$, let $q = \operatorname{lcm}(w_0, \ldots, w_n)$, and choose $r \gg 0$ such that  $\mcO_{\bfP(\bfw)}(rq)$ is very ample.  Let $m = rq$ and let $\nu: \bfP(\bfw) \to \bfP^N$ denote the corresponding Veronese embedding via $\mcO_{\bfP(\bfw)}(rq)$. 
Let $HF_X(z)$ and $HS_X(t)$ be the Hilbert function and Hilbert series of $X$, respectively, with the Hilbert function and series of $\nu(X)$ denoted similarly. 
We write $HS^{(m)}(t)$ for the series consisting of only the terms of $HS(t)$ where the power of $t$ is a multiple of $m$. 

Note that, as rational functions of $t$, we can write $HS^{(m)}(t) = \frac{1}{m}\sum_{j=0}^{m-1}HS(\zeta^jt)$ where $\zeta$ is a primitive $m^{th}$ root of unity (the root of unity filter). 
Since $HF_{\nu(X)}(z) = HF_X(mz)$, we have $HS_{\nu(X)}(t) = HS^{(m)}_X(t^{1/m})$.
In particular, we have 
\begin{align*}
P_{\nu(X)}(t) &= 
	\frac{(1-t)^{d+1}}{m}\sum_{j=0}^{m-1}\frac{P_X(\zeta^jt^{1/m})}{\prod_{i=0}^{d}(1-\zeta^jt^{w'_i/m})}
	   \\ &= 
	\frac{(1-t)^{d+1}}{m}\sum_{j=0}^{m-1}\frac{P_X(\zeta^jt^{1/m})\prod_{i=0}^{d}\prod_{\ell \ne j}(1-\zeta^{\ell}t^{w'_i/m})}		{\prod_{i=0}^{d}\prod_{\ell=0}^{m-1}(1-\zeta^{\ell}t^{w'_i/m})}\\	
	  & = 
	\frac{(1-t)^{d+1}}{m}\sum_{j=0}^{m-1}\frac{P_X(\zeta^jt^{1/m})\prod_{i=0}^{d}\prod_{\ell \ne j}(1-\zeta^{\ell}t^{w'_i/m})}{\prod_{i=0}^{d}(1-t^{w'_i})} \\
 	& = 
	\frac{1}{m}\sum_{j=0}^{m-1}\frac{P_X(\zeta^jt^{1/m})\prod_{i=0}^{d}\prod_{\ell \ne j}(1-\zeta^{\ell}t^{w'_i/m})}{\prod_{i=0}^{d}(1+ \cdots + t^{w'_i-1})}.
\end{align*}
Finally, setting $t=1$ and using the fact that $\prod_{\ell=1}^{m-1}(1-\zeta^\ell) = m$ and $\prod_{\ell \ne j}(1-\zeta^{\ell}) = 0$ if $j \ne 0$, we obtain
$$
\deg(X) = \frac{\deg(\nu(X))}{m^{d}} = \frac{P_Y(1)}{m^{d}} 
	= \left(\frac{1}{m}\frac{P_X(1)m^{d+1}}{w'_0w'_1\cdots w'_{d}}\right) \frac{1}{m^{d}} = \frac{P_X(1)}{w'_0w'_1\cdots w'_{d}}.
\qedhere $$
\end{proof}

\begin{remark}
Note that we can always (nonuniquely) reduce the Hilbert series to the necessary form. The Hilbert series of a $(d+1)$-dimensional module $M$ over $S(\bfw)$ can always be written as a rational function
\[
\frac{P'(t)}{\prod_{i=0}^n(1-t^{w_i})}
\]
which has a pole of order $d+1$ at $t=1$. Thus the numerator is divisible by $(1-t)^{n-d}$, and we may reduce the Hilbert series by cancelling out $n-d$ factors of $(1-t)$ from the denominator however we please. For instance, we can write
\[
HS_M(t) = \frac{P''(t)}{\prod_{i=1}^{n-d}(1 + t + \ldots + t^{w_i-1}) \prod_{i=n-d+1}^n(1-t^{w_i})},
\]
taking the rational function $P(t)$ in the lemma to be 
\[
P(t) = \frac{P''(t)}{\prod_{i=1}^{n-d}(1 + t + \ldots + t^{w_i-1})}.
\]
\end{remark}

\begin{example}  Let $\bfw=(2,3,5,5)$ and consider the graded ring $S(\bfw) = \kk[x,y,z_1,z_2]$. The ideal
$
(z_1- z_2 + xy, x^3z_1 + 2y^2z_1 + x^3z_2 - y^2z_2 )
$
defines an integral curve $X$ with  Hilbert series
$$
\frac{1-t^{11}}{(1-t^5)(1-t^3)(1-t^2)} = \frac{1+t+\cdots+t^{10}}{(1-t^5)(1-t^3)(1+t)}.
$$ 
Letting $P_X(t) =  \frac{1+t+\cdots + t^{10}}{t+1}$, we see that $P(1) = \frac{11}{2}$. By \Cref{thm_degree} we get that the degree of $X$ is $\frac{11}{30}$. 

Unlike the standard graded case, the choice of $P_X(t)$ is not unique. For instance, we could also reduce the Hilbert series as
\[
\frac{1+t+\cdots+t^{10}}{(1-t^3)(1-t^2)(1+t + t^2 + t^3 + t^4)}.
\]
In this case we take $P_X(t) =  \frac{1+t+\cdots + t^{10}}{1+t + t^2 + t^3 + t^4}$, which gives $P_X(t) = \frac{11}{5}$. Once again, we see that the degree of $X$ is $\frac{11}{30}$.
\end{example}

\section{Degree bounds in divisible weighted projective spaces}
In this section, we provide bounds for the degree of a non-degenerate subvariety of a divisible weighted projective space (\Cref{divisible_wps}). We also show how the classification of varieties of minimal degree differs from the classical case. We begin with some useful lemmas.

\begin{lemma} \label[lemma]{lemma_reducible} Let $\bfP(V) \subseteq \bfP(S_d)$ be a linear system of degree $d$ hypersurfaces in $\bfP^n$. If $\codim(\bfP(V)) \leq dn$, then $\bfP(V)$ contains a hypersurface that is a union of (possibly non-reduced) hyperplanes. 
\end{lemma}
\begin{proof} Let $\psi:\bfP(S_1)^d \to \bfP(S_d)$ denote the map $(g_1,g_2,\dots,g_d) \mapsto g_1g_2\cdots g_d$ and $\text{im}(\psi)$ the closure of the image. We need to show that $\bfP(V) \cap \text{im}(\psi) \ne \emptyset$. Since $\psi$ has finite fibers, the image of $\psi$ is a closed subscheme of dimension $nd$. The result follows. 
\end{proof}

\begin{lemma} \label[lemma]{lcm_very_ample} Let $\mathbf{P}(\bfw)=\bfP(1^{a_0},m_1^{a_1},\dots,m_k^{a_k})$ be a divisible weighted projective space. Then the line bundle $\mcO(m_k)$ is very ample on $\bfP(\bfw)$.
\end{lemma}
\begin{proof}
We will prove the stronger statement that $\mcO(m_k)$ is projectively normal, meaning a complete linear series defines a projectively normal variety. Since $S(\bfw)$ is normal, we need to show that $\text{Sym}^r(S(\bfw)_{m_k}) \to S(\bfw)_{rm_k}$ is surjective for  $r \geq 1$.

We proceed by induction on $k$, the number of distinct weights greater than $1$ in $\bfw$. When $k=0$, the line bundle is $\mcO(1)$ on $\bfP^{a_0-1}$, which is clearly projectively normal. Assume  $\mcO(m_{k-1})$ is projectively normal. To prove that $\mcO(m_k)$ is projectively normal, we must show that any monomial in $S(1^{a_0},m_1^{a_1},\dots,m_k^{a_k})$ of degree $rm_k$ with $r>1$ factors as a product of a degree $m_k$ monomial and a degree $(r-1)m_k$ monomial. Suppose $g = \prod_{i=0}^k\prod_{j = 1}^{a_i}x_{i,j}^{b_{i,j}}$ is such a monomial. If $b_{k,j}$ is nonzero for any $j$, then the degree $m_k$ variable $x_{k,j}$ may be factored out and we are done. If $b_{k,j} = 0$ for all $j$, then we have a degree $rm_k = rr'm_{k-1}$ monomial in $S(1^{a_0},m_1^{a_1},\dots,m_{k-1}^{a_{k-1}})$ which, by induction, factors completely into degree $m_{k-1}$ monomials, any $r'$ of which give a degree $m_k$ monomial that can be factored out of $g$.
\end{proof}

\begin{remark} \label[remark]{m_1va} 
Assume $\bfP(\bfw)$ is a divisible weighted projective space. Using the monomials of degree $m_1$, we can perform an $m_1$-th Veronese embedding on the weight $1$ variables to obtain the map
$$
\psi:\bfP(1^{a_0},m_1^{a_1},\dots,m_k^{a_k}) \to \bfP\left(1^{\binom{a_0-1 + m_1}{m_1}+a_1}, \left(\frac{m_2}{m_1}\right)^{a_2}, \dots,\left(\frac{m_k}{m_1}\right)^{a_k}\right).
$$
Note that the closed immersion given by $\mcO(m_k)$ (from \Cref{lcm_very_ample}) factors through $\psi$. Since the composition is a closed immersion, it follows that $\psi$ itself is a closed immersion.
\end{remark}

\begin{prop} \label{prop_degree_bound} Let $X$ be a non-degenerate subvariety of dimension $d$ in a divisible weighted projective space $\bfP(1^{a_0},m_1^{a_1},\dots,m_k^{a_k})$. Then 
\begin{align} \label{degree_bound_divisible}
\deg(X) \geq  \frac{1}{m_k^{d-1}}\left(a_0 -1+ \frac{a_1}{m_1} + \cdots + \frac{a_k}{m_k} + \frac{1-d}{m_k} \right).
\end{align}
Furthermore, let $i$ be the smallest index such that $d> \sum_{j=i+1}^k a_j$. If $i < k$, the bound can be improved to
\small\begin{align} \label{degree_bound_divisible_2}
\deg(X) \geq  
 \frac{1}{m_i^{d-\sum_{j=i+1}^ka_j - 1}m_{i+1}^{a_{i+1}}\cdots m_k^{a_k}}\left(a_0 -1+ \frac{a_1}{m_1} + \cdots + \frac{a_i}{m_i} + \frac{1+\left(\sum_{j=i+1}^ka_j-d\right)}{m_i} \right).
\end{align}  
\end{prop}
\begin{proof} 
We proceed by induction on $k$, the number of distinct weights greater than $1$. The base case of $k=0$ is classical; see \cite{Eisenbud_Harris_minimal} for a modern proof. Now assume $k\geq 1$, and let $X$ be a non-degenerate $d$-dimensional subvariety not satisfying \eqref{degree_bound_divisible}.
As in \Cref{m_1va}, consider the $m_1$-th Veronese embedding of the degree $1$-piece
$$
\psi:\bfP(1^{a_0},m_1^{a_1},\dots,m_k^{a_k}) \to \bfP\left(1^{\binom{a_0-1 + m_1}{m_1}+a_1}, \left(\frac{m_2}{m_1}\right)^{a_2}, \dots,\left(\frac{m_k}{m_1}\right)^{a_k}\right).
$$
By induction, the degree of any non-degenerate $d$-dimensional subvariety in $\bfP\left(1^{r}, \left(\frac{m_2}{m_1}\right)^{a_2}, \dots,\left(\frac{m_k}{m_1}\right)^{a_k}\right)$ is at least 
$$
\frac{m_1^{d-1}}{m_k^{d-1}}\left[r-1+ \frac{m_1a_1}{m_2} + \cdots + \frac{m_1(a_k+1-d)}{m_k}\right].
$$
On the other hand, by assumption, the degree of the image $\deg(\psi(X)) = m_1^d\deg(X)$ is less than 
\begin{align} \label{small_degree}
 \frac{m_1^{d-1}}{m_k^{d-1}}\left((a_0 -1)m_1+ a_1 + \frac{m_1a_2}{m_2} + \cdots + \frac{m_1(a_k+1-d)}{m_k} \right).
\end{align}
In particular, $\psi(X)$ must be degenerate, i.e., lies on a weighted hyperplane $V(f)$. By definition, the map $\psi^{\star}$ on coordinate rings leaves the variables of weight at least $2$  untouched. In particular, if the hyperplane $V(f)$ containing $\psi(X)$ had any coordinates with weight at least $2$, the $\psi^{\star}(f)$ would still have those coordinates, i.e., $X$ would lie on a weighted hyperplane, a contradiction. Thus, we may assume that $f$ is a linear combination of weight $1$ coordinates. In particular, $\psi(X)$ lies in the weighted projective subspace $\bfP\left(1^{\binom{a_0-1 + m_1}{m_1}+a_1-1}, \left(\frac{m_2}{m_1}\right)^{a_2}, \dots,\left(\frac{m_k}{m_1}\right)^{a_k}\right)$. We can keep repeating this argument until we obtain that $\psi(X)$ lies inside a linear subspace of the form $\bfP\left(1^{(a_0-1)m_1+a_1}, \left(\frac{m_2}{m_1}\right)^{a_2}, \dots,\left(\frac{m_k}{m_1}\right)^{a_k}\right)$. The pullback of the equations defining this linear space is a codimension $(a_0-1){m_1}$ subspace of $k[x_0,\dots,x_{a_0-1}]_{m_1}$. By \Cref{lemma_reducible}, this subspace must contain a union of hyperplanes, i.e., $X$ is degenerate, which is a contradiction. Thus, the degree of $X$ is given by \Cref{degree_bound_divisible}.

Now let $i$ be the smallest index such that $d> \sum_{j=i+1}^k a_j$ and assume that $i < k$. Choose a general form $f\in S(\bfw)$ of degree $m_k$. Then $X'=X \cap V(f)$  is a $(d-1)$-dimensional scheme in a smaller weighted projective space $\bfP(\bfw') \simeq \bfP(1^{a_0},m_1^{a_1},\dots,m_k^{a_k-1})$.
Consider the Veronese embedding $\nu$ of $\bfP(\bfw')$  given by $\mcO(m_k)$. Since we chose $f$ to have degree $m_k$, the image of $X'$ is $\nu(X)\cap H$ where $H$ is the hyperplane whose pullback is $V(f)$. Now 
\[
\deg(X') = \frac{1}{m_k^{d-1}}\deg(\nu(X')) = \frac{1}{m_k^{d-1}}\deg(\nu(X)\cap H) =  \frac{1}{m_k^{d-1}}\deg(\nu(X)) =  m_k\deg(X).
\]
We proceed inductively until we end up with a curve of minimal degree in $\bfP(1^{a_0},m_1^{a_1},\dots,m_{i-1}^{a_{i-1}},m_i^{d- \sum_{j=i+1}^k a_j})$. The degree of this was computed above, completing the proof. \qedhere
\end{proof}
 
\begin{definition} A subvariety $X \subseteq \mathbf{P}(\bfw)$ that attains the bound \Cref{degree_bound_divisible} or \Cref{degree_bound_divisible_2}, depending on its dimension, is called a \textbf{variety of minimal degree}.
\end{definition}

\begin{cor} \label[cor]{minimal_rational} Let $X$ be a non-degenerate  subvariety of minimal degree in a divisible weighted projective space. If $\dim(X) \leq a_k$, then $X$ is rational. Furthermore, if $\dim(X)=1$, then $X$ is also smooth.
\end{cor}
\begin{proof} Once again, we proceed by induction on $k$, the number of distinct weights greater than $1$. The base case of $k=0$ being classical  \cite{Eisenbud_Harris_minimal}. Following the notation and proof of \Cref{prop_degree_bound}, we see that $\psi(X)$ lies inside 
$$\bfP\left(1^{(a_0-1)m_1+a_1+1}, \left(\frac{m_2}{m_1}\right)^{a_2}, \dots,\left(\frac{m_k}{m_1}\right)^{a_k}\right),$$ 
and no smaller weighted linear subspace, i.e., it is non-degenerate in this weighted projective space. Since the degree is of $\psi(X)$ is given by \Cref{small_degree}, it is a variety of minimal degree. By induction, $\psi(X)$ is rational. Since $\psi$ is a closed immersion, $X$ is isomorphic to $\psi(X)$, which implies that $X$ is rational. The other statement also follows from induction, since a curve of minimal degree in $\bfP^n$ is smooth.
\end{proof}
 
The first substantive divergence from the classical case is that weighted varieties of minimal degree need not be determinantal; in fact they need not even be arithmetically Cohen--Macaulay.

\begin{example} \label[example]{three_curves_1}
Consider the following three maps $\bfP^1\to\bfP(1,1,2,2)$:
\begin{align*}
\phi_0\colon [s:t]\mapsto [s^2 : t^2 : s^3t : st^3 ]\\
\phi_1 \colon [s:t]\mapsto [s^2 : st : st^3 : t^4]\\
\phi_2 \colon [s:t]\mapsto [s^3t : st^3 : s^8 : t^8]
\end{align*}
with images $C_0, C_1,$ and $C_2$ respectively. 
The curves $C_1$ and $C_2$ are arithmetically Cohen--Macaulay and are cut out by the $2\times 2$ minors of the following two matrices, respectively
\[
\begin{pmatrix}
x_1 & x_2^2 & y_1\\
x_2 & y_1 & y_2
\end{pmatrix}
\quad\text{and}\quad
\begin{pmatrix}
y_1 & x_1& x_2^2 \\
 x_1^2& x_2 & y_2
\end{pmatrix}.
\]

Note that $C_1$ and $C_2$ cannot be obtained from each other via automorphisms of $\bfP(1,1,2,2)$. Indeed, we can consider how the curves meet the singular locus of $\bfP(1,1,2,2)$, which is $V(x_1,x_2)$. Note that $C_1$ intersects the singular locus in only one point $[0:0:0:1]$ (with multiplicity 2), while $C_2$ intersects the singular locus in two distinct points, $[0:0:0:1]$ and $[0:0:1:0]$.

On the other hand, $C_0$ is minimally generated by four binomials
$
x_2y_1-x_1y_2, x_1x_2^3-y_2^2, x_1^2x_2^2-y_1y_2$, and $x_1^3x_2-y_1^2$.
 In particular, it is not determinantal. To see that $C_0$ is not arithmetically Cohen–Macaulay, note that, for instance, using \cite{M2}, the minimal resolution of $S/I(C_0)$ is
 $$
 0 \xrightarrow{} S^1 
 \xrightarrow{\begin{pmatrix} y_2 \\ -y_1 \\ -x_2 \\ x_1 \end{pmatrix}} S^4 
 \xrightarrow {
 \begin{pmatrix} -y_1 & -y_2 & -x_1^2 & -x_1x_2^2 \\ 
 -x_2 & 0 & -y_2 &  0 \\
 x_1 & -x_2 & y_1 & -y_2 \\
 0 & x_1 & 0 & y_1 \\
 \end{pmatrix}} S^4 
 \xrightarrow{\small {\begin{pmatrix} x_2y_1-x_1y_2 \\ x_1^3x_2-y_1^2 \\ x_1^2x_2^2-y_1y_2\\ x_1x_2^3-y_1^2 \end{pmatrix}}^{T}} S^1 \xrightarrow{} S/I(C_0) \xrightarrow{} 0.
 $$
\end{example}

This example shows that classifying varieties of minimal degree, even at the level of curves, is more complicated than in the classical case, where all such curves are equivalent up to automorphisms of $\mathbf{P}^n$. In fact, a dimension count shows that there are often infinitely many automorphism classes of minimal degree curves in divisible weighted projective spaces.

\begin{example} \label[example]{dimension_count} By \Cref{prop_degree_bound}, the minimal degree of any non-degenerate curve in $\bfP(1,1,2^k)$ is $\frac{k+2}{2}$. In particular, any curve $C \subseteq \bfP(1,1,2^k)$ of minimal degree is the preimage of a rational normal curve of degree $k+2$ under the $2$-uple embedding $\psi:\bfP(1,1,2^k) \to \bfP^{k+2}$. 

Let $x_0,x_1,\dots,x_{k+2}$ denote the coordinates on $\bfP^{k+2}$.  The space of rational normal curves in $\bfP^{k+2}$ is a projective space of dimension $(k+3)^2 - 4$. For a rational normal curve to lie inside $\text{im}(\psi)$, it needs to satisfy $x_0x_2 - x_1^2 = 0$. By considering a general rational normal curve as a map $\phi: \bfP^1 \to \bfP^{k+2}$, the equation $x_0x_2-x_1^2$ turns into a degree $2(k+2)$ equation in $s,t$, the coordinates on $\bfP^1$. For the curve to lie inside $\text{im}(\psi)$, the coefficients of the equation have to be zero. Since there are $2k+5$ coefficients, the dimension of the space of rational normal curves in $\bfP(1,1,2^k)$ is at least $(k+3)^2 - 4 - (2k+5) = k^2+4k$.

On the other hand, the automorphism group of $\bfP(1,1,2^k)$ is 
$(\text{PGL}_2 \times \text{GL}_k)\rtimes (\text{Sym}^2(\kk^2))^{\oplus k}$. Explicitly, if the coordinates on $\bfP(1,1,2^k)$ are $x_1,x_2,y_1,\dots,y_k$, then the action is given by 
$$(x_1,x_2) \mapsto (Ax_1,Ax_2), \quad (y_1,\dots,y_k) \mapsto (By_1 + q_1(x_1,x_2),\dots,By_k + q_k(x_1,x_2))$$ 
with $A \in \text{GL}_2$, $B \in \text{GL}_k$ and $q_i \in \kk[x_1,x_2]_2$. The dimension of this group is $k^2+3k+3$. It follows that for $k>3$, the minimal degree curves in $\bfP(1,1,2^k)$ are not finite up to projective automorphisms.
\end{example}

\section{Weighted $1$-generic matrices and determinantal scrolls}

Having determined lower bounds for the degrees of non-degenerate weighted projective varieties, our next task is to explore the varieties that achieve this minimal degree. As we have just noted, the classification of all such varieties appears more complicated than in standard projective space. Rather than pursue a full classification, we construct and study a family of candidates analogous to rational normal scrolls. 

\subsection{pseudo 1-generic matrices}

\begin{definition}\label{def_gen_entries}
Let $M$ be a matrix with entries in $S(\bfw)$. A \textbf{generalized row} of a matrix $M$ with entries in $S$ is a nontrivial $\kk$-linear combination of the rows of $M$, and a \textbf{generalized entry} is a nontrivial $\kk$-linear combination of the entries of a generalized row. We say that the matrix $M$ is \textbf{pseudo 1-generic} if the entries of any generalized row are a regular sequence in $S(\bfw)_{(x_0,\ldots, x_n)}$. 
\end{definition}

\begin{remark} \label{change_coordinates} Note that if $M$ is pseudo 1-generic, then multiplication by an invertible matrix $P$ (that preserves homogeneity of the entries) preserves pseudo 1-genericity. Indeed, the entries of a generalized row are regular if and only if the Koszul complex on those entries is exact.  The Koszul complex on the corresponding generalized row of $MP$ is just the original complex multiplied by $P$ and is thus still exact.
\end{remark}

The ideal of maximal minors of $M$ is homogeneous if and only if $\deg(M_{i,j}) - \deg(M_{i',j})$ is constant in $j$. After reordering the rows and columns, we may assume that the degrees of the entries of $M$ are nondecreasing along both rows and columns. We can describe the degrees of the entries of $M$ entirely via the degrees of the entries in the first row and the first column. 
If the degrees in the first row are $(a_1, \ldots, a_q)$ and the degrees in the first column are $(a_1, a_1 + b_1, a_1+b_2, \ldots, a_1+b_{p-1})$, then we define the \textbf{profile} of $M$ to be $(a_1, \ldots, a_q; b_1, \ldots, b_{p-1})$. If all of the $b_i$ are zero, then we omit them and just write $(a_1, \ldots, a_q)$.

\begin{definition}\label{def_1generic}
Let $M$ be a pseudo 1-generic matrix over $S(\bfw)$. If $M$ has profile $(a_1, \ldots, a_q)$ which can be obtained by deleting entries from $(w_0, \ldots, w_n)$, then we say $M$ is \textbf{1-generic}.
\end{definition}

In the standard graded case, pseudo 1-generic is equivalent to 1-generic as in \cite[Section 6B]{eisenbud:syzygies}.

\begin{remark} \label{chop_first}
If there are $a_0$ many weights of the lowest degree $w_0$, then any $1$-generic matrix can have at most $a_0-1$ columns of weight $w_0$.
\end{remark}

\begin{prop} \label[prop]{thm_CM}
The vanishing locus of the maximal minors of a $p\times q$ pseudo 1-generic matrix over $S(\bfw)$ is an arithmetically Cohen--Macaulay scheme of codimension $q - p + 1$.
\end{prop}

\begin{proof}
Let  $M$ be a $p \times q$ pseudo 1-generic matrix over $S(\bfw)$, let $X = V(I_p(M)) \subset  \bfP(\bfw)$ be the defining scheme and $C(X) \subseteq \bfA(\bfw)$ its affine cone. 
Since $\codim(I_p(M)) \le q-p+1$, it suffices to show that $\dim (X) \le n-(q-p+1)$. 
Indeed, in this case, the Eagon-Northcott complex on $M$ would be a minimal free resolution of $S(\bfw)/I_X$ \cite[Thm A2.60]{eisenbud:syzygies}. 
It follows that $S(\bfw)/I_X$ has projective dimension $q-p+1$ and thus is Cohen--Macaulay.

We may assume that $X$ is not a cone, since adjoining variables preserves codimension and Cohen--Macaulayness. 
In particular, the rank of $M$ never drops to zero.  
A point $x \in \bfA(\bfw)$  lies in $X$ if and only if $\rank(M(x)) < p$, which is equivalent to some generalized row of $M$ vanishing at $x$. 
Note that the generalized rows are parametrized by $ \bfP^{p-1}$. 
Define the incidence variety
$$
\widetilde{X} = \left\{ (y,x) \in  \bfP^{p-1} \times  \bfA(w_0,\dots,w_n) : R_{y}(x)=0 \right\},
$$
where $R_{y}$ denotes the generalized row of $M$ corresponding to the parameter $y \in \bfP^{p-1}$. 
Since $M$ is pseudo $1$-generic, for any $y \in  \bfP^{p-1}$, the fiber $\widetilde{X}_y$ is defined by $q$ equations forming a regular sequence, hence
$\dim (\widetilde{X}_y) = n+1-q$.
It follows that $\dim (\widetilde{X}) = (n+1-q)+(p-1) = n+1-(q-p+1)$. 
Since the projection to $\bfA(w_0,\dots,w_n)$ maps $\widetilde{X}$ onto $C(X)$ we have $\dim(X) = \dim(C(X))-1 \leq \dim(\widetilde{X}) - 1 = n -(q-p+1)$, as required.
\end{proof}

\begin{remark} \label[remark]{XtildeCM} In fact, $\widetilde{X}$ in the proof of \Cref{thm_CM} is also arithmetically Cohen--Macaulay. The defining ideal of $\widetilde{X}$ is 
$$
\left(\sum_{i=1}^p y_i M_{i,1}, \sum_{i=1}^p y_i M_{1,2},\dots, \sum_{i=1}^p y_i M_{i,q} \right) \subseteq S(\bfw)[y_1,\dots,y_p].
$$
In particular, $\widetilde{X}$ is defined by $q$ equations. Since its dimension is also $n+1-(q-p+1) = (n+p)-q$, it follows that $\widetilde{X}$ is Cohen--Macaulay.
\end{remark}

In general, the maximal minors of a $1$-generic matrix might not define an irreducible or even a reduced variety. Moreover, this cannot always be detected by the profile, as the following examples show.

\begin{example} \label{profile_not_unique}
Let $S(\bfw) = \kk[x_1, x_2, x_3, y_1, y_2, y_3]$ with $\deg(x_i) = 1$ and $\deg(y_i) = 2$. Consider the matrices
\[ M = \begin{pmatrix} x_1 & x_2^2 & y_1 & y_2\\ x_2 & y_1 & x_3^2 & y_3\end{pmatrix} \quad \text{and} \quad
N = \begin{pmatrix} x_1 & x_2^2 & y_1 & y_2 \\ x_2 & x_3^2 & y_2 & y_3\end{pmatrix}.
\]
We will show that $M,N$ are both $1$-generic with the same profile $(1,2,2,2)$, but only $I_2(N)$ is prime.

To see that $M$ is pseudo 1-generic, take the ideal generated by a generalized row 
$$I:=\left(R_{[a:b]}\right) = (ax_1+bx_2,ax_2^2+by_1,ay_1+bx_3^2,ay_2+by_3).$$ 
Clearly, if $a=0$ or $b=0$, then $I $ is a complete intersection. 
If $b \ne 0$, then after a change of coordinates $I = (x_2,y_1,x_3^2,y_3)$ is a complete intersection. 
Thus $M$ is pseudo 1-generic and, since the profile of $M$ is $(1,2,2,2)$, it is also 1-generic. 
However, $I_2(M)$ is not prime since the factors of the minor $x_2^2x_3^2-y_1^2 = (x_2x_3-y_1)(x_2x_3+y_1)$ are not in $I_2(M)$.

Similarly, one can check that $N$ is also $1$-generic with profile $(1,2,2,2)$. 
Observe that $x_1$ is a non-zero divisor on $S(\bfw)/I_2(N)$; since $S(\bfw)/I_2(N)$ is Cohen--Macaulay, it suffices to check that the dimension drops by $1$. 
Thus, to show that $I_2(N)$ is a prime ideal it suffices to show that the localization $I_2(N)S_{x_1}$ is prime.  
As in \Cref{prime_is_generic}, the localized ideal is $I_2(N)S_{x_1} = (x_3^2 - \frac{x_2^3}{x_1}, y_2 - \frac{y_1x_2}{x_1}, y_3 - \frac{y_2x_2}{x_1})$ which is evidently a prime ideal.
\end{example}

Conversely, prime determinantal ideals need not be defined by pseudo 1-generic matrices.  
\begin{example}
Consider $\kk[x_1,x_2,y_1,\dots,y_7]$ with $\deg(x_i)=1$ and $\deg(y_i) = 2$ and the $3\times 3$ matrix
$$
M = 
\begin{pmatrix}
x_1^2 & x_1x_2 & y_1 \\
y_2 & y_3 & y_4 \\
y_5 & y_6 & y_7
\end{pmatrix}.
$$
It is easy to see that $\det(M)$ is irreducible. Thus, $I_3(M)$ is prime and Cohen--Macaulay, while the first row does not form a regular sequence.
\end{example}

However, if we restrict to matrices with two rows this phenomenon cannot occur. 

\begin{prop} \label[prop]{prime_is_generic} Let $M$ be a $2 \times n$ matrix such that $I_2(M) \subseteq S(\bfw)$ is homogeneous with expected codimension $n-1$. If $I_2(M)$ is also prime, then $M$ must be pseudo 1-generic.
\end{prop}
\begin{proof} Assume $Q := I_2(M)$ is a prime ideal. Since we can always perform row operations, it suffices to show that the last row is a regular sequence in $S(\bfw)_{\mfm}$. Let $M = (f_{i,j})_{i,j}$ be a non-zero matrix with $f_{i,j}$ homogeneous. If  $f_{2,1} = 0$, then $Q$ will not be prime.

If $n = 2$, then the bottom row fails to be a regular sequence exactly when $f_{2,2}$ is a zero divisor in $S(\bfw)_{\mfm}/(f_{2,1})$. This happens if $f_{2,1}$ and $f_{2,2}$ share a common factor; in particular, $I_2(M) = (\det(M))$ will not be a prime ideal, which is a contradiction.
So assume $n \geq 3$. Since $Q$ is graded, $f_{1,1}$ is not in $Q$ and, since it becomes a unit in $S_Q$, we can apply $S_Q$-invertible column operations to put $M$ into the form
$$
\begin{pmatrix}
1 & 0 & \cdots & 0\\
f_{2,1} & f'_{2,2} & \cdots & f'_{2,n} \\
\end{pmatrix}.
$$
In particular, $QS_Q = (f'_{2,2},\dots,f'_{2,n})$. Since $\text{ht}(Q) = n-1$, the elements $f'_{2,2},\dots,f'_{2,n}$ must form a regular sequence. This in turn implies that $f_{2,2},\dots,f_{2,n}$ form a regular sequence in $S(\bfw)_{\mfm}$ (at the level of Koszul complexes, we are just localizing and multiplying by an invertible matrix and both of these are faithfully flat operations). Since $f_{2,2},\dots,f_{2,n}$ are linearly independent modulo $Q^2S_Q$, so are $f'_{2,2},\dots,f'_{2,n}$.
In particular, $S_Q/QS_Q$ is a regular local ring and thus is a domain. Since $f_{2,1} \ne 0$, this implies that $f_{2,1}$ is a non-zero divisor in $S_Q/QS_Q= S_Q/(f_{2,2},\dots,f_{2,n})$. This implies that $f_{2,1}$ is a non-zero divisor modulo $(f_{2,2},\dots,f_{2,n})$, i.e., the bottom row is a regular sequence.
\end{proof}

The upshot of \cref{prime_is_generic} is that for $2\times n$ matrices, pseudo 1-generic is the right condition to characterize determinantal (irreducible) varieties. As mentioned earlier, increasing an entry of the profile of a matrix will increase the degree of the resulting determinantal variety. The ``minimal" profiles are those that belong to 1-generic matrices, so it is to these that we further restrict to define a weighted analogue of rational normal scrolls.

\begin{definition}\label[definition]{def_weightedScroll} 
Let $M$ be a 1-generic matrix in $S(\mathbf{w})$ with two rows. 
If $X = V(I_2(M)) \subseteq \mathbf{P}(\mathbf{w})$ is integral, we say that $X$  is a \textbf{weighted determinantal scroll}.  
\end{definition}

We build integrality into our definition because, in general,  the exact condition for integrality and smoothness depends on the arithmetic of polynomials that appear in the matrix (see \Cref{profile_not_unique} and \Cref{surface_possibilities}). Nonetheless, we show in \Cref{thm_minDegStructure} that the varieties of minimal degree defined by $1$-generic matrices are integral (and smooth), including the case of curves. 
We also include the qualifier ``determinantal" to account for the possibility that, from a more geometric standpoint, there are varieties that should be considered ``scrolls" that cannot be defined determinantally. Classically, a rational normal scroll is a family of linear spaces parameterized by $\bfP^1$, where the linear space corresponding to $[s:t]$ is the projective span of points corresponding to $[s:t]$ on a collection of rational curves. We have seen that not all minimal degree curves, though rational, are determinantal. Thus any geometric definition of a scroll akin to the classical case would be forced to include examples that are similarly not determinantal. 
Since we deal exclusively with determinantal scrolls, we may omit the qualifier for concision.

\begin{example} \label[example]{surface_possibilities}For simplicity, assume $\text{char}(\kk) = 0$. Let $X \subseteq \bfP(1^4,m^3)$ be the subscheme defined by maximal minors of
$$
M = 
\begin{pmatrix} 
x_1 & y_1 &  p_0 & y_3\\
x_2 & y_2 &y_3 + p_1 & p_2
\end{pmatrix}
$$
where $p_i \in \kk[x_1,\dots,x_4]$ are polynomials of degree $m$. Note that $M$ is $1$-generic iff $s^2p_0+stp_1-t^2p_2$ is not the zero polynomial for all $[s:t] \in \bfP^1$. To ensure it is integral we need to enforce the following conditions:
\begin{itemize}
\item $p_0,p_1,p_2$ do not share a common factor, and
\item $p_1^2 + 4p_0p_2$ is not a square.
\end{itemize}
Consider the following matrices when $m=2$:
$$
M = 
\begin{pmatrix} 
x_1 & y_1 & x_3^2 + x_2^2 & y_3\\
x_2 & y_2 & y_3 & x_4^2
\end{pmatrix}, \quad
M' = 
\begin{pmatrix} 
x_1 & y_1 & x_3^2 & y_3\\
x_2 & y_2 & y_3  & x_4^2
\end{pmatrix}, \quad
M'' = 
\begin{pmatrix} 
x_1 & y_1&  -x_3^2 & y_3 \\
x_2 & y_2 & y_3 + 2x_3x_4  & x_4^2
\end{pmatrix}.
$$
Let $X$, $X'$ and $X''$ be schemes defined by the maximal minors of the above matrices, respectively. We have the following facts about these varieties.
\begin{enumerate}
\item $X$ is integral with the general fiber of $X \to \bfP^1$ being isomorphic to $V(x_1^2+x_2^2+x_3^2) \subseteq \bfP(1^3,2)$.
\item  $X'$ is a reducible union. In particular, $I_2(M') $ is
$$
(y_3-x_3x_4,x_2x_3-x_1x_4,x_2y_1-x_1y_2,x_4y_1-x_3y_2) 
\cap (y_3+ x_3x_4,x_2x_3+x_1x_4,x_2y_1-x_1y_2,x_4y_1+x_3y_2).
$$
\item $X''$ is generically non-reduced. In particular, $\sqrt{I_2(M'')}$ is
$$ (y_3+x_3x_4,x_1x_4+x_2x_3,x_2y_1-x_1y_2,x_4y_1+x_3y_2).$$
\end{enumerate}
\end{example}

To describe the structure of $2\times n$ $1$-generic matrices, we appeal to the theory of $2\times n$ matrices of linear forms, specifically the Kronecker--Weierstrass characterization of matrix pencils. 

\begin{definition} \label[definition]{matrix_decomposition_notation} Let $\bfw = (1^{a_0}, m_1^{a_1},\dots,m_k^{a_k})$ and let $S(\bfw) = \kk[x_{i,j}]_{i,j}$ with $\deg(x_{i,j}) = m_i$ (see \Cref{setup}). Let $M$ be a $1$-generic $2 \times n$ matrix over $S(\bfw)$. 
Denote by $M_i$ the submatrix of $M$ corresponding to the columns of degree $m_i$.
If we consider the ring $S(m_i^{a_i}) := S(\bfw)/(\{x_{j,l}: j \ne i\})$,  then the restriction $\overline{M_i}$ to $S(m_i^{a_i})$ is a matrix whose entries are linear forms in the $x_{i,l}$ or zero. Now by the theory of Kronecker--Weierstrass forms (see \cite[Section 3]{RS24} and \cite{Chun90}), up to a linear change of coordinates in $S(m_i^{a_i})$, $\overline{M_i}$ is a concatenation of scroll blocks, Jordan blocks, nilpotent blocks and zero blocks. Lifting these back to $M_i$ and then applying a weighted change of coordinates to each degree $m_i$ variable just once, we see that $M_i$ is the concatenation of blocks, each of which has one of the following forms: 
\begin{enumerate}
\item a \textbf{weighted scroll block} 
$$
\begin{pmatrix} x_{i,j_1} & x_{i,j_2} & \cdots & x_{i, j_{l-2}} & x_{i,j_{l-1}} \\ 
x_{i,j_2} + p_2 & x_{i,j_3} + p_3 & \cdots & x_{i,j_{l-1}} + p_{l-1} & x_{i,j_{l}} 
 \end{pmatrix},
$$
\item a \textbf{weighted Jordan block}
$$\begin{pmatrix} x_{i,j_1} & x_{i,j_2} & \cdots  & x_{i,j_{l-1}} & x_{i,j_l} \\ 
x_{i,j_2} + \epsilon x_{i,j_1} + p_1  & x_{i,j_3} + \epsilon x_{i,j_2} + p_2  & \cdots  & x_{i,j_l} + \epsilon x_{i,j_{l-1}} + p_{l-1} & \epsilon x_{i,j_l} + p_l
 \end{pmatrix}
 $$
for some $\epsilon \in \kk$, 
\item a \textbf{weighted nilpotent block}
$$\begin{pmatrix} p_0 & x_{i,j_1} & x_{i,j_2} & \cdots & x_{i,j_{l-1}} & x_{i,j_l} \\ 
x_{i,j_1} +p_1& x_{i,j_2} +p_2 & x_{i,j_3} +p_3& \cdots & x_{i,j_l} + p_l & p_{l+1}
 \end{pmatrix}, 
 $$
\item or a \textbf{weighted zero block}
$$\begin{pmatrix} p_1 & p_2 & \cdots &  p_l \\ 
 p_{l+1}& p_{l+2} &  \cdots & p_{2l}
 \end{pmatrix}. 
$$
\end{enumerate}
In each of the above blocks,  $p_j \in S(1^{a_0},m_1^{a_1},\dots,m_{i-1}^{a_{i-1}})$ is a polynomial of degree $m_i$.
Moreover, the sets of variables of degree $m_i$ appearing in each block of $M_i$ are disjoint (though note the same lower-degree variable may appear in multiple blocks).
\end{definition}

\begin{remark} Two degenerate cases that appear are Jordan blocks of size one and nilpotent blocks of size one. These respectively look like 
$$
\begin{pmatrix} x_{i,j_1} \\ \epsilon x_{i,j_1} + p_1 \end{pmatrix},
\text{ and }
\begin{pmatrix} p_{1} \\ x_{i,j_1}  \end{pmatrix}.
$$
\end{remark}

Using this characterization, we now prove a structure theorem for $2\times n$ matrices with certain profiles. Note that we do not assume $S(\bfw)$ is divisible.

\begin{thm}\label[theorem]{thm_minDegStructure}
Let $M$ be a weighted $1$-generic  
matrix over $S(\bfw)$ with two rows having profile 
\[
(1^{a_0-1}, m_1^{a_1}, \ldots, m_{k-1}^{a_{k-1}}, m_{k}^{r_k}).
\]
Denote by $M_i$ the submatrix consisting of columns of degree $m_i$ so that the matrix $M$ can be written in block form as $M = (M_0 \mid M_1 \mid \cdots \mid M_k)$. Then, up to a change of coordinates,
\begin{itemize}
\item $M_0$ is a scroll block,
\item for each $1 \leq i \leq k-1$, $M_i$ is made up of Jordan blocks and at most one nilpotent block of size one,
\item $M_k$ is made up of Jordan blocks, scroll blocks and at most one nilpotent block of size one.
\end{itemize}
Furthermore, $I_2(M)$ is prime.
\end{thm}

\begin{remark}\label{rmk_max_wt}
The proof of the theorem holds also if the index $k$ is replaced by some $k'<k$. Indeed if this is the case, then none of the variables of degree $>m_{k'}$ appear in the matrix at all, so the subscheme is a cone over a scroll in a smaller weighted projective space.
\end{remark}

\begin{proof}
We use the notation introduced in \Cref{matrix_decomposition_notation}, and we note in particular that a degree $i$ weighted scroll block involving $\ell$ variables of degree $i$ has $\ell-1$ columns, a degree $i$ weighted Jordan block involving $\ell$ variables of degree $i$ has $\ell$ columns, and a degree $i$ weighted nilpotent block involving $\ell$ variables of degree $i$ has $\ell+1$ columns unless $\ell = 1$, in which case weighted nilpotent block may have just a single column.

Since $M_0$ has $a_0-1$ columns and involves only variables of degree $1$, it must consist of a single scroll block. After a change of coordinates we may assume
\begin{align} \label{scroll_block_1}
M_0 = \begin{pmatrix} x_{0,1} & x_{0,2} & \cdots & x_{0,a_0-2} & x_{0,a_0-1}\\ 
x_{0,2} & x_{0,3}& \cdots &x_{0,a_0-1}& x_{0,a_0}
 \end{pmatrix}.
\end{align}

Let $c^\text{scr}_1,c^\text{scr}_2,\dots,c^\text{scr}_{\alpha_1}$ be the sizes of all the scroll blocks of $M_1$, $c^{\text{jor}}_1,c^{\text{jor}}_2,\dots,c^{\text{jor}}_{\alpha_2}$ be the sizes of all the Jordan blocks of $M_2$, $c^{\text{nil}}_1,c^{\text{nil}}_2,\dots,c^{\text{nil}}_{\alpha_3}>1$ the sizes of nilpotent blocks of $M_1$ and $c^{\text{zer}}_1,c^{\text{zer}}_2,\dots,c^{\text{zer}}_{\alpha_4}$ the sizes of the zero blocks of $M_1$. Also, assume that there are $c^{\text{nil,1}}$ nilpotent blocks of size $1$. Because the number of columns in $M_1$ and the number of variables of degree $m_2$ are equal, we have 
\[
\sum_{\text{jor, scr, nil, zer}}\sum_i c_i^j = a_1\geq \alpha_1+\sum_i c_i^{\text{scr}} + \sum_i c_i^{\text{jor}} + \sum_i c_i^{\text{nil}} - \alpha_3 + c^{\text{nil,1}}.
\]
In particular, the total size of all weighted zero blocks appearing is $\sum_i c_i^{\text{zer}} \geq \alpha_1-\alpha_3+c^{\text{nil,1}}$. Suppose there were a weighted zero block in $M_1$. 
Using column operations, we see that the polynomials in the second row of the weighted zero block of $M_1$ are just multiples of $x_{0,1}^{m_1}$. Since $M$ is pseudo $1$-generic, so is $(M_0|M_1)$ and this implies that there is at most one column in the weighted zero block and its entry in the second row is $x_{0,1}^{m_1}$. Permuting columns we may assume that $M$ has a submatrix of the form
$$
\begin{pmatrix} 
 x_{0,1} & x_{0,2} & \cdots & x_{0,a_0-2} & x_{0,a_0-1} & p \\ 
x_{0,2} & x_{0,3}& \cdots &x_{0,a_0-1}& x_{0,a_0} & x_{0,1}^{m_1}
 \end{pmatrix}.
$$
Write $p = dx_{0,a_0}^{m_1} + q$ where $\deg_{x_{0,a_0}}(q) < m_1$. If $d=0$ then the elements of the first row will not form a regular sequence, contradicting the pseudo 1-genericity of $M$. So assume $d \ne 0$ and 
consider the generalized row $R_{[1:\frac{1}{c}]}$. We will show that there exists a $c$ such that the entries $R_{[1:\frac{1}{c}]}$ do not form a regular sequence.
The ideal of the generalized row is
$$(x_{0,1}+c^{-1}x_{0,2},\ x_{0,2}+c^{-1}x_{0,3},\  \dots ,\  x_{0,a_0-1} + c^{-1}x_{0,a_0},\ p + c^{-1}x_{0,1}^{m_1})$$
Starting with $i=a_0$ and decrementing by one, we apply the change of coordinates $x_{0,i} \mapsto c(x_{0,i} - x_{0,i-1})$ to get the ideal 
$$(x_{0,2},x_{0,3}, \dots , x_{0,a_0}, \widetilde{p} + c^{-1}x_{0,1}^{m_1}).$$
The above is a regular sequence exactly when $\widetilde{p}$ does not contain the monomial $-c^{-1}x_{0,1}^{m_1}$. If we just look at coefficient of  $x_{0,1}^{m_1}$ in $\widetilde{p}$, as a polynomial in $c$, it looks like $\pm dc^{(a_0-1)m_1} + \text{lower order terms}$ since after the change of coordinates, each $x_{0,i}$ appearing in $p$ contributes to a $\pm c^{i-1}x_{0,1}$ in $\widetilde{p}$ and this contribution is maximized by the monomial $dx_{0,a_0}^{m_1}$. 
In particular, the coefficient of $x_{0,1}^{m_1}$ is a non-zero polynomial $f(c)$ over $\kk$. 
If $f(c) = -c^{-1}$, then $\widetilde{p}$ contains the term $-c^{-1}x_{0,1}^{m_1}$, in which case $M$ fails to be pseudo 1-generic. 
But $f(c) = -c^{-1}$ if and only if $f(c)c+1=0$, and the latter has a non-zero solution in $\kk$. 
Thus, $M$ is not pseudo 1-generic, which is a contradiction. Thus we may assume that $M_1$ has no weighted zero blocks. This implies $\alpha_3-c^{\text{nil,1}} \geq \alpha_1$, that is the number of scroll blocks is at most the number of nilpotent blocks of size at least two.

Now suppose there were a nilpotent block of size at least two, and that after permuting columns $M$ had a submatrix of the form
$$
\begin{pmatrix} 
 x_{0,1} & x_{0,2} & \cdots & x_{0,a_0-2} & x_{0,a_0-1} & p_0 & x_{1,1} & x_{1,2} & \cdots & x_{1,l-1} & x_{1,l} \\ 
x_{0,2} & x_{0,3}& \cdots &x_{0,a_0-1}& x_{0,a_0} & x_{1,1} + p_1 & x_{1,2} + p_2 & x_{1,3} + p_3 & \cdots & x_{1,l}+ p_{l} & p_{l+1}
 \end{pmatrix}
$$
with $p_i \in S(1^{a_0})$. 
Starting with $i = l+1$ and decrementing by $1$, use the first $a_0-1$ columns to subtract off everything except a power of $x_{0,1}$ from each $p_i$ in the second row and then apply a weighted change of coordinates to make the corresponding entry in the first row a variable. This gives us the submatrix
$$
\begin{pmatrix} 
 x_{0,1} & x_{0,2} & \cdots & x_{0,a_0-1} & p & x_{1,1} & x_{1,2} & \cdots & x_{1,l-1} & x_{1,l} \\ 
x_{0,2} & x_{0,3}& \cdots & x_{0,a_0} & x_{1,1} + c_{1}x_{0,1}^{m_1}& x_{1,2} + c_{2}x_{0,1}^{m_1} & x_{1,3} + c_{3}x_{0,1}^{m_1} & \cdots & x_{1,l}+ c_{l}x_{0,1}^{m_1}& c_{l+1}x_{0,1}^{m_1}
 \end{pmatrix}.
$$
We will once again show that this is not pseudo 1-generic by exhibiting a generalized row whose entries fail to be a regular sequence. 
Write $p = dx_{0,a_0}^{m_1} + q$ where $\deg_{x_{0,a_0}}(q) < m_1$. If $d=0$ then the first row itself will not form a regular sequence. Pseudo 1-genericity is similarly violated if $c_{l+1} \ne 0$. Consider the generalized row $R_{[1:\frac{1}{c}]}$ and, starting at $i = l$ and decrementing by one until $i=2$, apply the change of coordinates $x_{1,i} \mapsto c(x_{1,i} - x_{1,i-1}) - c_{i}x_{0,1}^{m_1}$. Then starting with $i=a_0$ and decrementing by one until $i=2$, apply the change of coordinates $x_{0,i} \mapsto c(x_{0,i} - x_{0,i-1})$. Since the nilpotent block has size at least two, this leaves us with the generalized row 
$$
\left(
x_{0,2},\ x_{0,3},\ \dots ,\ x_{0,a_0},\
c^{-1}x_{1,1} + ( c_{1}c^{-1} - f_1(c))x_{0,1}^{m_1},\
x_{1,2},\ \dots,\ x_{1,l},\  
(c^{-1}c_{l+1}+f_2(c))x_{0,1}^{m_1} -x_{1,1}
\right).
$$
Just as before, $f_1(c)$ is a non-zero polynomial in $c$ of degree $m_1a_0$ and $f_2(c)$ is linear in $c$. This is a regular sequence if and only if the matrix 
$$
\begin{pmatrix}
c^{-1}  & c_1c^{-1} -f_1(c) \\
-1 & c^{-1}c_{l+1} + f_2(c)
\end{pmatrix}
$$
is invertible. Since $f_1$ is non-zero,
the determinant of this matrix is a non-zero Laurent polynomial of degree $a_0+l-1$ in $c$ (note that $a_0 \geq 1$ and  $m_1 \geq 2$). In particular, there is some $c$ for which the matrix is not invertible, again contradicting pseudo 1-genericity. Thus, there are no nilpotent blocks of size at least $2$, which in turn implies that there are no scroll blocks in $M_1$.

So assume the nilpotent block had size one. If there were another nilpotent block, the first row would not form a regular sequence. Thus, $M_1$ is made up of Jordan blocks and possibly one nilpotent block of size one. So $(M_0 \mid M_1)$ looks like 
$$
\begin{pmatrix} 
 x_{0,1} & x_{0,2} & \cdots & x_{0,a_0-2} & x_{0,a_0-1}    
 & \text{Jordan} & \text{nilpotent}\\ 
x_{0,2} & x_{0,3}& \cdots &x_{0,a_0-1}& x_{0,a_0} 
& \text{blocks} & \text{block of size 1}
 \end{pmatrix}.
$$
Since the number of columns of $M_1$ is equal to the number of variables of degree $m_1$, every variable of degree $m_1$ appears in $M_1$ and thus appears in any generalized row of $M$. In particular, there is a change of coordinates leaving the variables of degree $\ne m_1$ fixed such that, after a change of coordinates, the generalized row of $M$ contains the variables of degree $m_1$. It follows that $(M_0 \mid M_1 \mid M_2)$ is pseudo 1-generic iff $(M_0 \mid M_2)$ is pseudo $1$-generic. By induction, we see that each  $M_i$ is made up of Jordan blocks and possibly one single nilpotent block for $1\leq i \leq k-1$. Similarly, the induction shows that $M_k$ cannot have a weighted zero blocks and nilpotent blocks of size two or more. However, since we don't require $M_k$ to have $a_k$ columns, we can also have scroll blocks.

We have shown that $M$ is made up of scroll blocks, nilpotent blocks of size one and Jordan blocks. We will now use this fact to show that $I_2(M)$ is prime. With notation as in \Cref{thm_CM}, consider the map $\pi:\widetilde{X} \to \bfP^1$ and note that by \Cref{XtildeCM},  $\widetilde{X}$ is Cohen--Macaulay. Since $\pi$ has equidimensional fibers, and maps a Cohen--Macaulay scheme to a smooth scheme, miracle flatness implies that $\pi$ is flat. In particular, $\pi$ is an open map. If $a,b \ne 0$, the ideal defined by the generalized row $R_{[a:b]}$ is projectively equivalent to a an ideal defined by variables, and thus is irreducible. This implies that $\pi^{-1}(y)$ is irreducible for $y \ne [0:1],[1:0]$. Since $\pi$ is an open map and it follows that $\widetilde{X}$ is irreducible \cite[Tag 004Z]{stacks_project} and, in particular, its image in $\bfA(\bfw)$, which is $X$, is also irreducible. Thus, $\sqrt{I}$ is a prime ideal. Since $S/I$ is Cohen--Macaulay, $I$ is a primary ideal.

After adding a multiple of the second row to the first and changing coordinates, the structural statement we have already shown allows us to assume that the first row of $M$ is made of distinct variables.
Let $Q \subseteq S(\bfw)$ be the ideal generated by the entries of the first row and note that it is prime. Let $I := I_2(M)$ and note that  $I \subseteq Q$. We will now show that $IS(\bfw)_{Q}$ is a prime ideal. Since $x_{0,a_0} \notin Q$, $x_{0,a_0}$ is a unit in $I_2(M)S(\bfw)_{Q}$ and thus we can apply invertible column operations to put $M$ into the form
$$
\begin{pmatrix}
p_1 & \cdots & p_{l-1} & p_{l} \\
0 & \cdots & 0 	& 1
\end{pmatrix}.
$$
Note that $IS_Q  = (p_1,\dots,p_l)S_Q$. Since the first row of $M$ was linearly independent modulo $Q^2S_{Q}$, so is the first row of the above matrix. Thus $S_Q/(p_1,\dots,p_n) = S_Q/IS_Q$ is a regular local ring and, in particular, a domain. Therefore $IS_Q$ is a prime ideal and since $I$ is a primary ideal, $I$ must also be prime. 
\end{proof}

We will prove in \cref{noncone_degree} that the scrolls with profile as in \cref{thm_minDegStructure} have minimal degree. Though we defer this result, we will use some consequences of it to further explore the one-dimensional case.

\begin{cor} \label{cor_smooth_curves} Let $M$ be a $1$-generic $2\times r$ matrix over $S(\bfw)$ such that $X = V(I_2(M))$ is one-dimensional. Then $X$ is integral and smooth. 
\end{cor}
\begin{proof}
By \Cref{thm_CM}, $I_2(M)$ is Cohen--Macaulay so its codimension in $S(\bfw)$ is $r-1$. Keeping with the notation in \Cref{matrix_decomposition_notation}, since $S(\bfw)/I_2(M)$ is two-dimensional, we get that $r -1 = \sum a_i - 2$. By \Cref{chop_first}, the profile must be  $(1^{a_0-1}, m_1^{a_1},\dots,m_k^{a_k})$ and thus \Cref{thm_minDegStructure} implies that $X$ is integral. By \Cref{noncone_degree}, $X$ has minimal degree and the result follows from \Cref{minimal_rational}.
\end{proof}

The structural portion of the proof of \Cref{thm_minDegStructure} allows us to explicitly describe the blocks appearing in the matrix $M$; after a weighted change of coordinates, we assume that $M$ is of the following Kronecker--Weierstrass form. As in \cref{rmk_max_wt}, it is safe to replace all instances of the index $k$ in the following with some $k'<k$, but it suffices to prove for $k$.

\begin{definition}[{A Kronecker--Weierstrass form}] \label{KW_form_simplified}
Let $M = (M_0\mid M_1\mid\cdots\mid M_k)$ be a $1$-generic matrix with profile $(1^{a_0-1}, m_1^{a_1}, \ldots, m_{k-1}^{a_{k-1}}, m_{k}^{r_k})$.  

For $1 \leq i \leq k$ let
$$
N_i = 
\begin{pmatrix}
q_{i} \\
x_{i,1}
\end{pmatrix}.
$$
For $1 \leq i \leq k$, let $n_i$ denote the number of Jordan blocks appearing in $M_i$. For $1\leq j\leq n_i$, write
$$
J(\epsilon_{i,j})= 
\small
\begin{pmatrix} 
x_{i,a_{i,j}+1} & \cdots & x_{i,a_{i,j+1}-1} & x_{i,a_{i,j+1}} \\ 
x_{i,a_{i,j}+2} + \epsilon_jx_{i,a_{i,j}+1} + p_{a_{i,j}+1} & 
 \cdots & 
x_{i,a_{i,j+1}}+ \epsilon_jx_{i,a_{i,j+1}-1} + p_{i,a_{i,j+1}-1} & 
\epsilon_jx_{i,a_{i,j+1}} + p_{i,a_{i,j+1}}
 \end{pmatrix}
$$
where $0 \leq a_{i,1} < a_{i,2} < \cdots < a_{i, n_i} < a_{i,{n_{i}+1}}$, $a_{i,1} \in\{0,1\}$, and $a_{i,{n_i+1}} = a_i$ for all $i\neq k$.

Let $\sigma$ denote the number of scroll blocks in $M_k$ and for $1\leq j\leq \sigma$, write 
$$
S(j) = 
\small
\begin{pmatrix} 
x_{k,a_{k,n_k+j}+1} & \cdots & x_{k,a_{k,n_k+j+1}-2} & x_{k,a_{k,n_k+j+1}-1} \\ 
x_{k,a_{k,n_k+j}+2} + p_{a_{k,n_k+j}+1} & 
 \cdots & 
x_{k,a_{k,n_k+j+1}} + p_{k,a_{k,n_k+j+1}-1} & 
x_{k,a_{k,n_k+j+1}} 
 \end{pmatrix}
$$
where $ a_{k,n_k+1} < \cdots < a_{k,n_k+\sigma-1} < a_{k,{n_{k}+\sigma}} =  r_k$. 

We may also assume that 
\begin{itemize}
\item $p_{i,\star} \in \text{span}\{(x_{\ell,a_{\ell,j}+1})^{m_i/m_{\ell}}:  1 \leq \ell \leq i-1\text{ and } j \geq 1 \}$, and
\item $q_{i,1} \in \text{span}\{x_{\ell,a_{\ell,1}}^{m_i/m_{\ell}}: 1 \leq \ell \leq i-1\} \cup \{x_{\ell,1}^{m_i/m_{\ell}}: m_\ell \mid m_i \text{ and } 2 \leq \ell \leq i-1\}$.
\end{itemize}

The matrix $M =(M_0\mid M_1\mid\cdots\mid M_k)$ is in \textbf{Kronecker--Weierstrass form} if
\begin{enumerate}
\item The first block is
$$
M_0 = 
\begin{pmatrix} 
 x_{0,1} & x_{0,2} & \cdots & x_{0,a_0-2} & x_{0,a_0-1} \\ 
x_{0,2} & x_{0,3}& \cdots &x_{0,a_0-1}& x_{0,a_0} 
 \end{pmatrix}.
$$
\item For $1\leq i\leq k-1$, if $a_{i,1} = 0$ 
$$
M_i = 
\begin{pmatrix}
J(\epsilon_{i,1})&  \cdots &  J(\epsilon_{i,n_i})
\end{pmatrix}
$$  
and if $a_{i,1} = 1$
$$
M_i = 
\begin{pmatrix}
N_i & J(\epsilon_{i,1})&  \cdots &  J(\epsilon_{i,n_i})
\end{pmatrix}.
$$  
\item If $a_{k,1} = 0$,
$$
M_k =
\begin{pmatrix}
J(\epsilon_{k,1}) & \cdots & J(\epsilon_{k,n_k}) & S(1) & \cdots & S(\sigma)
\end{pmatrix}
$$
and if $a_{i,1} = 1$
$$
M_k = 
\begin{pmatrix}
N _k &J (\epsilon_{k,1}) & \cdots & J(\epsilon_{k,n_k}) & S(1) & \cdots & S(\sigma)
\end{pmatrix}.
$$  
\end{enumerate}
\end{definition}

\begin{remark} If $S(\bfw)$ is a divisible weighted projective space, there exists a $1$-generic matrix $M$ with profile $(1^{a_0-1}, m_1^{a_1}, \ldots, m_{k-1}^{a_{k-1}}, m_{k}^{r_k})$ for every valid sequence $a_{i,j}$ appearing in \Cref{KW_form_simplified}. That is, a $1$-generic matrix exists for any valid collection of Jordan blocks (with prescribed $\epsilon_{i,j}$ values), nilpotent blocks of size one, and scroll blocks in degree $m_k$.
\end{remark}

\begin{remark} \label[remark]{not_proj_equiv}The Kronecker--Weierstrass form is not determined up to projective equivalence. For example, in $\bfP(1^2,2^2)$ , consider the ideal defined by the $2\times2$ minors of the following $1$-generic matrix:
$$
\begin{pmatrix}
x_1 & y_1 & y_2 \\
x_2 & y_2 + \epsilon y_1 & \epsilon y_2 + x_1^2 
\end{pmatrix}.
$$
The ideal is independent of $\epsilon \in \kk$.
More generally,  one can perform a row operation followed by a change of coordinates to always assume that one of the $\epsilon_{i,j}$ parameters equals zero. 
\end{remark}

The curves defined by $1$-generic matrices are rational \cite[Corollary 6.9(b)]{eisenbud:syzygies}, which naturally leads to the question of their parameterization. It would be beneficial to define these parameterizations globally; however, because these curves lie in a weighted projective space, the morphism from $\mathbf{P}^1$ is not always determined by a single line bundle and sections in the same manner as maps to $\mathbf{P}^n$. However, we can instead provide a map from a stacky curve of genus zero to the weighted projective \emph{stack} $\mcP(\bfw)$, whose image under the coarse space morphism to $\bfP(\bfw)$ is our minimal degree curve. To do this, we use the root stack construction. For background on the root stack construction, we refer the reader to \cite{Olsson2016AlgebraicSA}; we introduce here only the notation necessary to state and prove our result. 
We remark that the use of stack-theoretic machinery is restricted to the remainder of this section. This application serves as a friendly invitation to the utility of stacks in describing concrete phenomena that arise quite naturally in nonstandard graded commutative algebra.

\begin{notation}
For a curve $C$ and an effective Cartier divisor $D$, we denote by $\sqrt[r]{(C,D)}$ the \textbf{$r^{th}$ root stack} along $D$. The result of this construction is essentially to introduce stacky structure along $D$, where the points of $D$ are replaced by stacky points of order $r$. If $r$ is invertible and $C$ is smooth, the result is a tame Deligne--Mumford stack with coarse space $\pi\colon \sqrt[r]{(C,D)} \to C$. For us, a crucial feature of this construction is the existence of a line bundle $\mcL = \sqrt[r]{\mcO_C(D)}$ on $\sqrt[r]{(C,D)}$ such that $\mcL^{\otimes r} = \pi^*\mcO_C(D)$. The root stack construction can be iterated; that is, for disjoint effective divisors $D,D'$, we can form the stack $\sqrt[r]{(C,D)}\times_C\sqrt[r']{(C,D')}$ whose coarse space is $C$, and where stacky structure of order $r$ is introduced along $D$ while stacky structure of order $r'$ is introduced along $D'$. We will only invoke the case where $C = \bfP^1$ and each divisor $D$ is a single point.  
\end{notation}

\begin{prop} \label[prop]{stacky_parameterization} Let $\bfP(\bfw)$ be a divisible weighted projective space. Let $C \subseteq \bfP(\bfw) $ be a curve cut out by the minors of a 1-generic matrix with two rows, and let $\mathcal{C}$ be the preimage of $C$ in $\mathcal{P}(\bfw)$, under the coarse space morphism.
Then there is an injective morphism from a smooth stacky curve to $\mathcal{C}$.
\end{prop}

Before we give the proof we first demonstrate with an example in hopes that it will help to guide the reader through the notation required in the general case.

\begin{example}\label{ex_stack_parameterization} Consider the curve $C \subseteq \bfP(1^3,2^3,4^3)$ whose associated $1$-generic matrix has exactly one Jordan block in degrees 2 and 4:
\[
\begin{tikzpicture}
\matrix (M) [matrix of math nodes,
             left delimiter=(, right delimiter=),
             column sep=6pt, row sep=6pt,
             nodes={anchor=center}] 
{
x_{0,1} & x_{0,2} & x_{1,1} & x_{1,2}                 & x_{1,3} & x_{2,1}                                              & x_{2,2}                                             & x_{2,3} \\
\vphantom{x_{0,1}^2} x_{0,2} & \vphantom{x_{0,1}^2}x_{0,3} &  \vphantom{x_{0,1}^2}x_{1,2} & \vphantom{x_{0,1}^2}x_{1,3}  &  x_{0,1}^2& x_{2,2} + x_{2,1} + x_{1,1}^2 & x_{2,3} +  x_{2,2}  &x_{2,3} +  x_{0,1}^4 + x_{1,1}^2\\
};

\coordinate (b12) at ($(M-2-2.east)!0.5!(M-2-3.west)$);
\coordinate (b34) at ($(M-2-5.east)!0.5!(M-2-6.west)$);

\draw[dashed] (b12 |- M-1-1.north) -- (b12 |- M-2-1.south);
\draw[dashed] (b34 |- M-1-1.north) -- (b34 |- M-2-1.south);

\draw[decorate,decoration={brace,mirror,amplitude=6pt}]
  (M-2-1.south west) -- (M-2-2.south east)
  node[midway,below=8pt] {$M_0$};
\draw[decorate,decoration={brace,mirror,amplitude=6pt}]
  (M-2-3.south west) -- (M-2-5.south east)
  node[midway,below=8pt] {$M_1$, $(\epsilon=0)$};
\draw[decorate,decoration={brace,mirror,amplitude=6pt}]
  (M-2-6.south west) -- (M-2-8.south east)
  node[midway,below=8pt] {$M_2$, $(\epsilon=1)$};

\end{tikzpicture}
\]
Naively, we attempt to parameterize $C$ using sections of a line bundle on $\bfP^1$ where the coordinates are $[s:t]$. Examining the $M_0$ block, we define the degree $1$ variables using sections of $L = \mcO_{\bfP^1}(2)$, setting $x_{0,1} = s^2$, $x_{0,2} = st$, and $x_{0,3} = t^2$.

We now seek sections of powers of $L$ to parameterize the degree $2$ and $4$ variables while satisfying the determinantal relations forced by $M$. 
For the degree $2$ variables, we start with the relation $x_{0,1}^3 = x_{0,2}x_{1,3}$ and then recursively obtain $x_{1,3} = s^5t^{-1}$, $x_{1,2} = s^{6}t^{-2}$ and $x_{1,1} = s^{7}t^{-3}$. Moving recursively through the degree $4$ variables, we find
\begin{align*}
x_{2,3}  &= \frac{x_{0,1}(x_{0,1}^4 + x_{1,1}^2)}{x_{0,2}-x_{0,1}}
= \frac{s^{10} + s^{16}t^{-6}}{(st-s^{2})}
= \frac{s^{9}t^{6}+s^{15}}{t^{6}(t-s)} \\
x_{2,2}  &= \frac{x_{0,1}(x_{2,3})}{x_{0,2}-x_{0,1}}
=\frac{x_{0,1}^6+x_{0,1}^2x_{1,1}^2}{(x_{0,2}-x_{0,1})^2}
=\frac{s^{12}+s^{18}t^{-6}}{(st-s^{2})^{2}}
=\frac{s^{10}t^6+s^{16}}{t^{6}(t-s)^{2}} \\
x_{2,1} &=  \frac{x_{0,1}(x_{2,2} + x_{1,1}^2)}{x_{0,2}-x_{0,1}}
= \frac{x_{0,1}\left(x_{0,1}^6 + x_{0,1}^2x_{1,1}^2+x_{1,1}^2(x_{0,2}-x_{0,1})^2\right)}{(x_{0,2}-x_{0,1})^3}
 =\frac{s^2 \left( s^{12}t^6 + s^{18} + s^{14}(st - s^2)^2 \right)}{t^6 (st - s^2)^3}\\
& \hspace{9.6cm} = \frac{\left( s^{11}t^6 + s^{17} + s^{15} (t - s)^2 \right)}{t^6 (t - s)^3}.
\end{align*}
In order to get an actual map, we need to clear denominators. In order to clear the $(t-s)^3$ from the denominator of $x_{2,1}$, we scale all degree $4$ variables by $(t-s)^3$. This in turn scales all degree $2$ variables by $(t-s)^{3/2}$ and all degree $1$ variables by $(t-s)^{3/4}$. To resolve this issue, we pass from $\bfP^1$ to the root stack $\sqrt[4]{(\bfP^1, [1:1])}$ and let $v$ be a local section of the root line bundle $\sqrt[4]{\mcO_{\bfP^1}([1:1])}$ so that $v^4 = t-s$. Now we can happily replace all instances of $(t-s)^{1/4}$ by $v$. In the same vein, in order to remove the $t^6$ from the denominator of all of the degree $4$ variables, we scale each one by $t^6$, simultaneously scaling the degree $2$ variables by $t^3$ and the degree 1 variables by $t^{3/2}$. As before, we handle the root by passing to $\mcX = \sqrt[4]{(\bfP^1, [1:1])}\times_{\bfP^1}\sqrt[2]{(\bfP^1, [1:0])}$, letting $u$ be a local section of $\sqrt[2]{\mcO_{\bfP^1}([1:0])}$ so that $u^2 = t$.

Let $L'$ be the line bundle $\pi^*\left(\mcO_{\bfP^1}(2)\right)\otimes_\mcX \left(\sqrt[2]{\mcO_{\bfP^1}([1:0])}\right)^3 \otimes_\mcX \left(\sqrt[4]{\mcO_{\bfP^1}([1:1])}\right)^3$. The weighted series
\[
\{ s^2u^3v^3,\ stu^3v^3,\ t^2u^3v^3,\ s^7v^6,\ s^6tv^6,\ s^5t^2v^6,\ s^{11}t^6+s^{17} + s^{15}v^8,\ (s^{10}t^6+s^{16})v^4,\ (s^{9}t^6+s^{15})v^8\}
\]
defines a map
$
\mcX\to \mcP(1^3,\ 2^3,\ 4^3)
$
whose image, after composing with the coarse space morphism, is $C$.
\end{example}

\begin{proof}
Write $M$ in Kronecker--Weierstrass form, so that $M = (M_0 \mid M_1 \mid \cdots \mid M_k)$ and \[M_i = \begin{pmatrix}N_i & J(\epsilon_{i,1}) & \cdots & J(\epsilon_{i,n_i})\end{pmatrix} \text{ or } \begin{pmatrix} J(\epsilon_{i,1}) & \cdots & J(\epsilon_{i,n_i})\end{pmatrix}\] as in \Cref{KW_form_simplified}. First, assume that all of the $\epsilon_{i,j}$ are distinct. We will define a map to $\mcC$ from the root stack 
$
\sqrt[m_k]{(\bfP^1, [0:1])}\times_{\bfP^1} \sqrt[m_1]{(\bfP^1,[1:\epsilon_{1,1}])}\times_{\bfP^1}\dots\times_{\bfP^1} \sqrt[m_1]{(\bfP^1, [1:\epsilon_{1,n_1}])}\times_{\bfP^1} \dots \times_{\bfP^1}\sqrt[m_k]{(\bfP^1, [1:\epsilon_{k,n_k}])}.
$
Denote this stack by $\mcX$ and let $[s:t]$ be the coordinates on $\bfP^1$, its coarse space. To begin, we set $x_{0,i} = s^{a_0-i}t^{i-1}$, which clearly satisfies all determinantal relations coming from $M_0$. We work block by block, starting at the tail of each Jordan block and working our way  to the beginning, enforcing the determinantal relations. For a block $J(\epsilon_{i,j})$, the determinantal relations force
\[
x_{i,a_{i,j+1}} = \frac{x_{0,1}p_{a_{i,j+1}}}{x_{0,2}-\epsilon_{i,j}x_{0,1}}
\]
and for $1\leq \ell<a_{i,j+1}-a_{i,j}$ we have
\begin{align*}
x_{i,a_{i,j+1}-\ell} &= \frac{x_{0,1}}{x_{0,2}-\epsilon_{i,j}x_{0,1}}(x_{i,a_{i,j+1}-\ell+1} + p_{a_{i,j+1}-\ell})\\
&= \frac{x_{0,1}}{(x_{0,2}-\epsilon_{i,j}x_{0,1})^{\ell+1}} \sum_{b=0}^\ell (x_{0,2}-\epsilon_{i,j}x_{0,1})^b x_{0,1}^{\ell-b} p_{a_{i,j+1}-b}\\
&=\frac{1}{(s^{a_0-2})^{\ell+1}(t-\epsilon_{i,j}s)^{\ell+1}}\sum_{b=0}^\ell s^{a_0(\ell+1)-b}(t-\epsilon_{i,j}s)^{b}p_{a_{i,j+1}-b}\\
&=\frac{1}{(t-\epsilon_{i,j}s)^{\ell+1}}\sum_{b=0}^\ell s^{\ell+1-b}(t-\epsilon_{i,j}s)^{b}p_{a_{i,j+1}-b}
\end{align*}
where $p_{a_{i,j+1}-b}(s,t)$ has degree $m_i(a_0-1)$. We repeat this for all Jordan blocks. If a nilpotent block $N_i$ appears, we similarly set 
\[
x_{i,1} = \frac{tq_i}{s}.
\]

For each $\epsilon_{i,j}$, let $v_{i,j}$ be a local section of the root bundle $\sqrt[m_i]{\mcO_{\bfP^1}([1:\epsilon_{i,j}])}$ so that $v_{i,j}^{m_i} = (t-\epsilon_{i,j}s)$. We scale all the degree $m_i$ variables by 
$v_{i,j}^{(a_{i,j+1}-a_{i,j})m_i}$; 
the exponent here is just the size of the Jordan block $J(\epsilon_{i,j})$ multiplied by $m_i$. By homogeneity, this scales the degree $m_{i'}$ variables by 
$
v_{i,j}^{(a_{i,j+1}-a_{i,j})m_{i'}}.
$
This clears all powers of $(t-\epsilon_{i,j}s)$ from the denominators of all variables.
Note that this procedure will guarantee that for each $\epsilon_{i,j}$ there is a degree $m_i$ variable that is not divisible by $v_{i,j}$. 

After repeating for all $\epsilon_{i,j}$, each variable now has a denominator consisting only of powers of $s$ coming from the variables $x_{i,1}$ and polynomials thereof.
For each index $i$, let $\mu_i$ be the highest power of $s$ that can be factored out of the denominator of any of the $x_{i,\star}$ variables. Pick an index $i$ such that $\frac{\mu_i}{m_i} \geq \frac{\mu_{i'}}{m_i'}$ for all indices $i'$. 
Let $u$ be a local section of the root bundle $\sqrt[m_k]{\mcO_{\bfP^1}([0:1])}^{m_k/m_i}$ so that $u^{m_i} = s$.
To clear denominators, we scale the degree $m_i$ variables $x_{i,\star}$ by $s^{\mu_{i}} = u^{m_i\mu_i}$. By homogeneity, this scales the degree $m_{i'}$ variables by $u^{\mu_{i}m_{i'}}$. This clears powers of $s$ from every denominator. The choice of $\mu_i$ also guarantees the existence of at least one variable that is not divisible by $s$.

By construction, there is no common factor among all of the $x_{i,j}$. We have thus parameterized the $x_{i,j}$ by a basepoint free weighted series for the line bundle $L = \pi^*(\mcO(a_0))\otimes L'$ where $L'$ is a tensor product of powers of the root line bundles $\sqrt[m_k]{\mcO_{\bfP^1}([0:1])}$ and $\sqrt[m_i]{\mcO_{\bfP^1}([1:\epsilon_{i,j}])}$. This defines a map from $\mcX\to \mcP(\bfw)$ whose image after composing with the coarse morphism is $V(I_2(M))$.

In case not all of the $\epsilon_{i,j}$ are distinct, we must be slightly more careful not to introduce common factors. However, we can clear powers of $(1-\epsilon_{i,j}t)$ from denominators minimally as we did for powers of $s$ in order to ensure that the resulting weighted series is basepoint free.
This map is injective on geometric points by construction, as the variables $x_{0,1}, x_{0,2}$ completely recover $s$ and $t$.
\end{proof} 

Since weighted $1$-generic curves are smooth (\Cref{cor_smooth_curves}), it is natural to ask if the parameterization above is a closed immersion into $\mcP(\bfw)$. Unfortunately, this is usually not the case. In fact, the preimage of the $1$-generic curve in the weighted projective stack is usually a singular stack and the map in \Cref{stacky_parameterization} is a resolution of singularities. 

\begin{example}\label{ex_sing_stacky_curve} Consider the integral curve $C \subseteq \bfP(1,1,2,4)$ defined by
$$ I =
I_2
\begin{pmatrix}
x_1 & y_1  & z_1 \\
x_2 & x_1^2 & y_1^2
\end{pmatrix}
$$ 
Let $\mathcal{C} \subseteq \mcP(1,1,2,4)$ be its preimage. To see that $\mathcal{C}$ is singular, we consider the open substack $[U/\mathbf{G}_m]$ where $U = \text{Spec}(S_{z_1}/I)$. The quotient stack $[U/\mathbf{G}_m]$ is singular since 
$$
(S/I)_{z_1}  \simeq \frac{\kk[x_1,x_2,y_1,z_1^{\pm1}]}{(x_1^3-y_1x_2,x_1\frac{y_1^2}{z_1}-x_2,\frac{y_1^3}{z_1}-x_1^2)} \simeq \frac{\kk[x_1,y_1,z_1^{\pm 1}]}{(\frac{y_1^3}{z_1}-x_1^2)}
$$
defines a cuspidal singularity. 

Note that \Cref{stacky_parameterization} gives a parameterization $\mcP(1,4) \to\mathcal{C}$ via
$
[s:t] \mapsto [s^3t:s^7:s^2t^3:t^7].
$
It is not hard to show that the morphism is not a closed immersion since it does not separate tangent vectors. 
\end{example}

We end this section with two remarks on the structure of higher dimensional weighted scrolls. 

\begin{example}  As the dimension of the scroll increases, all the different blocks in \Cref{matrix_decomposition_notation} can appear. The number of blocks of a given type are subject to constraints analogous to the ones given in \Cref{thm_minDegStructure}. For an explicit example, here is a four dimensional weighted scroll in $\mathbf{P}(1^5,m^7,(mn)^3)$ whose $1$-generic matrix has a weighted Jordan, nilpotent and scroll blocks in degree $m$, and a Jordan block in degree $mn$:
\[
\begin{tikzpicture}
\matrix (M) [matrix of math nodes,
             left delimiter=(, right delimiter=),
             column sep=6pt, row sep=6pt,
             nodes={anchor=center}] 
{
 x_1 & y_1 & y_2 & y_4 & y_5 & x_4^m & y_6 & y_7 & z_1 & z_2 & z_3 \\
\vphantom{x_1^{m}} x_2 & \vphantom{x_1^{m}} y_2 & \vphantom{x_1^{m}} y_3 & \vphantom{x_1^{m}} y_5 & \vphantom{x_1^{m}} x_3^m & \vphantom{x_1^{m}} y_6 & \vphantom{x_1^{m}} y_7 & x_5^m &\vphantom{x_1^{m}} z_2+z_1 & \vphantom{x_1^{m}} z_3+z_2 & z_3+x_1^{mn} \\
};

\coordinate (b12) at ($(M-2-1.east)!0.5!(M-2-2.west)$);
\coordinate (b34) at ($(M-2-3.east)!0.5!(M-2-4.west)$);
\coordinate (b56) at ($(M-2-5.east)!0.5!(M-2-6.west)$);
\coordinate (b89) at ($(M-2-8.east)!0.5!(M-2-9.west)$);

\draw[dashed] (b12 |- M-1-1.north) -- (b12 |- M-2-1.south);
\draw[dashed] (b34 |- M-1-1.north) -- (b34 |- M-2-1.south);
\draw[dashed] (b56 |- M-1-1.north) -- (b56 |- M-2-1.south);
\draw[dashed] (b89 |- M-1-1.north) -- (b89 |- M-2-1.south);

\draw[decorate,decoration={brace,mirror,amplitude=6pt}]
  (M-2-1.south west) -- (M-2-1.south east)
  node[midway,below=8pt] {$M_0$};
\draw[decorate,decoration={brace,mirror,amplitude=6pt}]
  (M-2-2.south west) -- (M-2-3.south east)
  node[midway,below=8pt] {scroll};
\draw[decorate,decoration={brace,mirror,amplitude=6pt}]
  (M-2-4.south west) -- (M-2-5.south east)
  node[midway,below=8pt] {Jordan};
\draw[decorate,decoration={brace,mirror,amplitude=6pt}]
  (M-2-6.south west) -- (M-2-8.south east)
  node[midway,below=8pt] {nilpotent};
\draw[decorate,decoration={brace,mirror,amplitude=6pt}]
  (M-2-9.south west) -- (M-2-11.south east)
  node[midway,below=8pt] {Jordan};

\end{tikzpicture}
\]
\end{example}

\begin{example}\label{ex_sing_scroll}  If we are interested in smooth determinantal scrolls, we usually want to avoid weighted zero blocks. For example, consider the  surface scroll $X$ in $\bfP(1^3,2^2)$ defined by
$$I_2
\begin{pmatrix}
x_1 & x_2^2 & y_1 \\
x_2 & x_3^2 & y_2
\end{pmatrix}.
$$ 
Consider the smooth affine open $U = D(x_1) \simeq \mathbf{A}^4$. The coordinate ring of $X|_{U}$ is
$$
\frac{\kk[x_2,x_3,y_1,y_2]}{(x_3^2-x_2^3, y_2-y_1x_2 , y_2x_2^2-y_1x_3^2 )} 
= \frac{\kk[x_2,x_3,y_1,y_2]}{(x_3^2-x_2^3,y_2-y_1x_2, y_2x_2^2 - y_1x_3^2)}
= \frac{\kk[x_2,x_3,y_1]}{(x_3^2-x_2^3)}.
$$ 
In particular, $X|_U$ is a cone over a cuspidal curve.
\end{example}

\section{Degree and syzygies of determinantal scrolls in divisible spaces}
In this section, we compute the degree of a determinantal scroll in a divisible weighted projective space. Using this, we show that the minimal degree bound obtained in \Cref{prop_degree_bound} is sharp, and we classify the scrolls that attain this bound. We also compute the regularity and the “linearity” of the free resolution. The results of this section rely on a Hilbert series computation of the homogeneous coordinate ring of a weighted determinantal scroll via the Eagon–Northcott complex.

\subsection{Degrees of determinantal scrolls}
\begin{prop} \label[prop]{lemma_scroll_hs} Let $\bfP(\bfw) = \bfP(1^{a_0}, m_1^{a_1},\ldots, m_k^{a_k})$ be a divisible weighted projective space. The Hilbert series of a determinantal scroll $X \subseteq \bfP(\bfw)$ with profile $(1^{r_0},m_1^{r_1},\dots,m_k^{r_k})$ is
\begin{align} \label{HilbertSeries}
HS(t) = \frac{\sum_{i=0}^k r_it^{m_i}\left(1+t^{m_i}+t^{2m_i} + \cdots + t^{\left(\frac{m_k}{m_i}-1\right)m_i} \right) + 1- t^{m_k}}{\left(\prod_{j=0}^{k-1}(1-t^{m_j})^{a_j-r_j}\right)(1-t^{m_k})^{a_k-r_k+1}}.
\end{align}
\end{prop}

\begin{proof} 
Let $M$ be the $1$-generic matrix whose maximal minors cut out $X$. By \cite[Thm A2.60]{eisenbud:syzygies}, the minimal free resolution of $S/I_X$ is the Eagon-Northcott complex associated to the map
$F:= \bigoplus_{i=0}^k S(-m_i)^{\oplus r_i} \overset{M}{\rightarrow} G:= S^2$. This complex is of the form 
$
0 \rightarrow E_{r} \rightarrow \cdots \rightarrow E_2 \rightarrow E_1 \rightarrow S \rightarrow 0
$
where the module $E_i$ is equal to $(\Sym_{i-1}(G))^\star\otimes \bigwedge^{i+1}F $ which expands as
$$ 
S^i \otimes 
	\left(\bigoplus_{s_1,\dots,s_k} \bigwedge^{i+1- \sum s_i} S^{r_0}(-1) \otimes \bigwedge^{s_1} S^{r_1}(-m_1) \otimes \cdots \otimes \bigwedge^{s_k} S^{r_k}(-m_k) \right).
	$$
From this we see that $\beta_{i,j}$ is $i$ times the coefficient of $t^jx^{i+1}$ in $\Phi(x,t) = \prod_{\ell=0}^k (1-t^{m_\ell}x)^{r_\ell}$ for $i \geq 1$. On the other hand, we can use the multinomial expansion of $\Phi$ to obtain,
\begin{align*}
\sum_i  (-1)^i \left(\sum_j \beta_{i,j}t^j\right)x^{i} & = 1 + \sum_i (-1)^i \left(\sum_j i[t^jx^{i+1}] \Phi(x,t)t^j\right)x^i \\
 & = 1+\sum_i (-1)^i i \left(\sum_{s_1,\dots,s_k} 
\binom{r_0}{i+1-s_1-\cdots -s_k}
\binom{r_1}{s_1}\cdots \binom{r_k}{s_k}t^{i+1-\sum s_\ell + \sum s_\ell m_\ell}\right)x^{i} \\
&= -\frac{\partial}{\partial x} \Phi(x,t) + \Phi(x,t)   \\
&=  -\sum_{i=0}^kr_it^{m_i}(1-t^{m_i}x)^{r_i-1}\prod_{j \ne i}(1-t^{m_j-1}x)^{r_j} + \prod_{j}(1-t^{m_j}x)^{r_j} \\
& = \prod_{j}(1-t^{m_j}x)^{r_j-1}\left(\sum_{i=0}^k r_it^{m_i}\prod_{j \ne i}(1-t^{m_j}x) + \prod_{j}(1-t^{m_j}x)\right). 
\end{align*}
Then the Hilbert series of the scroll is 
$$
\frac{\left[-\frac{\partial}{\partial x} \Phi(x,t) + \Phi(x,t)\right]_{x=1}}{\prod_{j}(1-t^{m_j})^{a_j}} = 
\frac{\sum_{i=0}^k r_it^{m_i}\prod_{j \ne i}(1-t^{m_j}) + \prod_{j}(1-t^{m_j})}{\prod_{j}(1-t^{m_j})^{a_j-r_j+1}}.
$$
Using the fact that $m_i\mid m_{k}$ the numerator can be written as
\small
$$
\sum_{i=0}^k r_it^{m_i}\prod_{j \ne i}(1-t^{m_j}) + \prod_{j}(1-t^{m_j}) = 
\prod_{j=0}^{k-1}(1-t^{m_j}) \left(\sum_{i=0}^k r_it^{m_i}\left(1+t^{m_i}+t^{2m_i} + \cdots + t^{\left(\frac{m_k}{m_i}-1\right)m_i} \right) + (1-t^{m_k})\right).
$$
In conclusion, the reduced Hilbert series of the scroll is
$$
\frac{\sum_{i=0}^k r_it^{m_i}\left(1+t^{m_i}+t^{2m_i} + \cdots + t^{\left(\frac{m_k}{m_i}-1\right)m_i} \right) + 1- t^{m_k}}{\left(\prod_{j=0}^{k-1}(1-t^{m_j})^{a_j-r_j}\right)(1-t^{m_k})^{a_k-r_k+1}}. \qedhere
$$
\end{proof}

\begin{cor} \label[cor]{theorem_degree_scroll} 
Let $\bfP(1^{a_0}, m_1^{a_1}, \ldots, m_k^{a_k})$ be a divisible weighted projective space and $X$ a determinantal scroll with profile $(1^{r_0},m_1^{r_1},\dots,m_k^{r_k})$.   The degree of $X$ is 
$$
\frac{\sum_{i=0}^k \frac{r_i}{m_i}}{\prod_{j=0}^{k}{m_j}^{a_j-r_j}}.
$$
\end{cor}
\begin{proof} 
First, assume that $X$ is not a cone. Since the codimension of $X$ is $\sum_{j=0}^k r_j - 1$, its dimension is $\sum_{j=0}^k (a_j - r_j)$. The result then follows from \Cref{thm_degree} and \Cref{lemma_scroll_hs}. 
Now notice that taking the cone with respect to a weight $m_i$ coordinate will increase $a_i$ while leaving $r_i$ unchanged. This multiplies the denominator in the above formula by $m_i$. By \Cref{cone_rem}, we see that this is precisely the degree of the $m_i$-cone.
\end{proof}

\begin{thm} \label[thm]{noncone_degree}  \label[thm]{cone_degree} Let $\bfP(\bfw) = \bfP(1^{a_0}, m_1^{a_1}, \ldots, m_k^{a_k})$ be a divisible weighted projective space and $X$ a $d$-dimensional determinantal scroll.
\begin{enumerate}
\item If $d \leq a_k$, then $X$ has minimal degree iff its profile is $(1^{a_0-1},m_1^{a_1},\dots,m_{k-1}^{a_{k-1}},m_k^{a_k+1-d})$. In particular, its degree is  
$$
\frac{1}{m_k^{d-1}}\left(a_0 -1+ \frac{a_1}{m_1} + \cdots + \frac{a_k}{m_k} + \frac{1-d}{m_k} \right).
$$
\item If $d > a_k$, then $X$ has minimal degree iff its profile is $(1^{a_0-1},m_1^{a_1},\dots,m_{i-1}^{a_{i-1}},m_{i}^{(a_{i}+\cdots+a_k)+1-d})$ where $i$ is the smallest index for which $d> \sum_{j=i+1}^k a_j$. In particular, its degree is 
\small \begin{align*} 
 \frac{1}{m_i^{d-\sum_{j=i+1}^ka_j - 1}m_{i+1}^{a_{i+1}}\cdots m_k^{a_k}}\left(a_0 -1+ \frac{a_1}{m_1} + \cdots + \frac{a_i}{m_i} + \frac{1+\sum_{j=i+1}^ka_j-d}{m_i} \right).
\end{align*}  
\end{enumerate} 
\end{thm}
\begin{proof} The profile $(1^{r_0},m_1^{r_1},\dots,m_k^{r_k})$ of a $d$-dimensional determinantal scroll satisfies $d = \sum_{j=0}^k (a_j - r_j)$, $0 \leq r_0 \leq a_0-1$, and $0 \leq r_j\leq a_j$ for larger $j$. By \Cref{theorem_degree_scroll}, we can write the degree of the scroll as $C(\sum_{j=0}^k \frac{r_j}{m_j})\prod_{j=0}^k m_j^{r_j}$, for a constant $C$ only depending on the ambient projective space. Thus, we need to minimize the function
 $$
 f(r_0,\dots,r_k) = \left(\sum_{j=0}^k \frac{r_j}{m_j}\right)\prod_{j=0}^k m_j^{r_j}
$$
subject to the constraints $0 \leq r_0 \leq a_0 -1$ and $0 \leq r_j \leq a_j$. We also require the dimension $d$ (equivalently the sum $\sum_{j=0}^k r_j$) to be fixed. 

We can minimize $f$ greedily. In particular, we have $f(r_0,\dots,r_{i-1},r_i+1,r_{i+1}-1,r_{i+2},\dots,r_k) \leq f(r_0,\dots,r_k)$. Indeed, let $\omega = \sum_{j=0}^k \frac{r_j}{m_j}$ and note that 
$$
\frac{f(r_0,\dots,r_{i-1},r_i+1,r_{i+1}-1,r_{i+2},\dots,r_k)}{f(r_0,\dots,r_k)}
= \frac{\omega + \frac{1}{m_i} - \frac{1}{m_{i+1}}}{\omega} \cdot \frac{m_i}{m_{i+1}}
$$
The desired inequality now follows from 
$m_{i}(\omega + \frac{1}{m_i} - \frac{1}{m_{i+1}}) 
\leq m_i\omega + 1 - \frac{m_i}{m_{i+1}}
\leq m_{i+1}\omega$.
Thus, to minimize $f$, we first maximize $r_0$, then maximize $r_1$, then maximize $r_2$ until the constraint on $\sum r_j$ forces us to stop. In other words, let $i$ be the smallest index such that $d > \sum_{j=i+1}^ka_j$. Then we take $r_0 = a_0-1$, $r_j = a_j$ for all $1 \leq j \leq i-1$ and  $r_{i} = \sum_{j=i}^{k}a_j+1-d$. Finally, plugging this into the formula in \Cref{theorem_degree_scroll} gives the desired degree.
\end{proof}

\begin{remark}\label{rmk_minDegCone} 
\Cref{theorem_degree_scroll} also implies that a cone over a minimal degree scroll is itself of minimal degree only when the weight of the new variable is at least as large as all pre-existing weights.
More precisely, consider a minimal degree scroll $X$ in a divisible weighted projective space $\bfP(1^{a_0}, m_1^{a_1}, \ldots, m_k^{a_k})$. A cone over $X$ in $\bfP(1^{a_0}, m_1^{a_1}, \ldots, m_k^{a_k}, m_{k+1})$ is of minimal degree if $m_k \mid m_{k+1}$. In contrast, a cone over $X$ in $\bfP(1^{a_0}, m_1^{a_1}, \ldots, m_i^{a_i+1}, \ldots, m_k^{a_k})$ with $i < k$ will not be of minimal degree.\end{remark}

We have already encountered the first clear difference between weighted determinantal scrolls and classical rational normal scrolls in projective space: not all weighted determinantal scrolls have minimal degree. That said, every divisible weighted projective space does contain a minimal degree variety that is either a weighted determinantal scroll or a cone over one. However, as seen in \Cref{three_curves_1}, the weighted determinantal scrolls we have described do not account for all varieties of minimal degree. 
Similarly, as we have already mentioned, we expect that there are varieties of minimal degree which should be considered ``scrolls" from a geometric standpoint but which are not determinantal.

\subsection{Syzygies of determinantal scrolls}\label{section_regularity}

Our next goal is to explore the relationship between the degree of a weighted projective variety and its syzygies, using determinantal scrolls as a family of test cases. We focus on two main measures of the simplicity of syzygies: the ``linearity" of the resolution as captured by a weighted version of the $N_p$ conditions introduced by \cite{Brown_Erman_Positivity}, and the regularity of the defining ideal. 

Before defining the weighted versions of these measurements, we first recall the definitions in the standard case. Let $M$ be a module over $S$ with homogeneous maximal ideal $\mfm$, and let $F_\bullet\to M$ be its minimal free resolution. Then $M$ satisfies the \textbf{$N_0$ condition} if $H^1_\mfm(M) = 0$, and $M$ satisfies the \textbf{$N_p$ condition} if $F_i$ is generated in degrees $\leq i+1$ for all $1\leq i\leq p$. The $N_p$ conditions measure how many steps the maps in $F_\bullet$ are linear.
The \textbf{(Castelnuovo--Mumford) regularity} of $M$ is the least integer $r$ such that $H^i_\mfm(M)_j = 0$ for all $i\geq 0$ and $j>r-i$. Equivalently, it is equal to the greatest $d_i-i$, where $d_i$ is the largest degree of a generator of $F_i$. 

It is not clear what should replace the notion of “linearity” in a nonstandard graded ring. One approach is to compare the degrees of the maps in a free resolution with those in the Koszul complex on the variables. Adopting this viewpoint, if we wish to treat all variables as “linear,” then, in order to read off the linearity of the resolution from its Betti numbers alone, we must also regard as “linear” any polynomial whose degree coincides with that of a variable. This idea is formalized in the following definition, first introduced in \cite[Definition 1.1]{Brown_Erman_Positivity}.

\begin{definition}[{\textbf{Weighted $N_p$ conditions}}] \label[definition]{def_wNk}
For a weight sequence $\bfw = (w_0 \leq \ldots \leq w_n)$, write $w^i = \sum_{j=0}^{i-1} w_{n-j}$, i.e., $w^i$ is the sum of the $i$ largest weights. Let $X\subset \bfP(\bfw)$  and suppose $\bfF_\bullet$ is a minimal graded free $S(\bfw)$-resolution of the coordinate ring $S(\bfw)/I_X$. We say that $X$ satisfies the condition
\begin{itemize}
\item $wN_0$  if $H^1_\mfm(S(\bfw)/I_X) = 0$, 
\item $wN_p$ if $X$ satisfies $wN_0$ and $F_i$ is generated in degrees $\leq w^{i+1}$ for $1\leq i\leq p$.
\end{itemize}
\end{definition}

By virtue of being Cohen--Macaulay, all determinantal scrolls automatically satisfy $wN_0$. To determine if they satisfy the $wN_p$ condition in general, we compute their Betti numbers. 

\begin{notation}  For each $i>0$, let \textbf{$\tau_i(X)$} denote the largest $j$ for which $\beta_{i,j}(X)$ is nonzero.
\end{notation} 

\begin{lemma}\label[lemma]{lemma_max_betti}
Let $\bfP(\bfw)$ be a divisible weighted projective space and $X$ a determinantal scroll with profile $\bfu = (1^{r_0},m_1^{r_1},\dots,m_k^{r_k})$. Then for $0<i\leq \left(\sum_{j=0}^k r_j\right) -1$, 
 we have $\tau_i = u^{i+1}$
where $u^{i+1}$ denotes the sum of the greatest $i+1$ entries of $\bfu$.
\end{lemma}
\begin{proof} As in \Cref{lemma_scroll_hs}, the Betti numbers $\beta_{i,j}$ of $X$ are given by the coefficients of $t^jx^{i+1}$ in $\prod_{\ell = 0}^k(1-t^{m_\ell}x)^{r_\ell}$. From this we see that $\tau_i$ is equal to the maximum over tuples $(s_0, \ldots, s_k)$ of the quantity $i+1+\sum_{\ell=0}^k s_\ell(m_\ell-1)$ subject to $s_\ell\leq r_\ell$ for each $0\leq \ell\leq k$ and $\sum_{\ell=0}^ks_\ell= i+1$. This can be maximized greedily by taking $s_k$ as large as possible, then $s_{k-1}$, and so on in descending order until $\sum s_\ell = i+1$. This yields a maximum of $i+1 + \sum_{\ell = 0}^k s_\ell(m_\ell -1)  = \sum_{\ell = 0}^k s_\ell m_\ell= u^{i+1}$, the sum of the largest $i+1$ values in the profile.
\end{proof}

\begin{cor}\label{cor_scrolls_wNk}
Determinantal scrolls in divisible weighted projective spaces satisfy the $wN_p$ conditions for all $p$.
\end{cor}
\begin{proof} 
Since a weighted determinantal scroll with profile $\bfu$ is given by a $1$-generic matrix, we have 
$u^{i+1} \leq w^{i+1}$ and the result follows from \Cref{lemma_max_betti}.
\end{proof}

This parallels the classical setting, where rational normal scrolls in standard projective space satisfy the $N_p$ conditions. However, a significant departure in the weighted setting is that the $wN_p$ conditions allow for a much broader range of Betti tables than their standard graded counterparts. Specifically, the $wN_p$ conditions impose less rigid restrictions on the overall shape of the resolution. 

\begin{example}\label{ex_scrolls_wNk}
Consider $\bfP(1^2, 3^2, 6^3)$ and let $X_1$ be a determinantal scroll with profile $(1,3,6^3)$, $X_2$ a determinantal scroll with profile $(3^2,6^3)$, and $X_3$ a determinantal scroll with profile $(1,3^2,6^2)$. In \Cref{fig_bettiTables}, we display the Betti tables for each, highlighting the last entry of each column in the Betti table that may be nonzero in order for the relevant $wN_p$ condition to be satisfied. We can see that $X_2$ is in a sense ``maximally $wN_p$" for all $1\leq p\leq 4$ in that each column of the Betti table attains the maximum allowable height, while $X_1$ attains this maximum in columns $1$, $2$, and $3$, and $X_3$ attains it only in column $1$. 
\end{example}
\begin{figure}[H]
\scriptsize 
\renewcommand{\arraystretch}{0.84} 
\setlength{\arraycolsep}{4pt}

\begin{minipage}{.3\textwidth}
\[
\begin{matrix}
       & 0 & 1 & 2 & 3 & 4\\
      0: & 1 & . & . & . & .\\
      1: & . & . & . & . & .\\
      2: & . & . & . & . & .\\
      3: & . & 1 & . & . & .\\
      4: & . & . & . & . & .\\
      5: & . & . & . & . & .\\
      6: & . & 3 & . & . & .\\
      7: & . & . & . & . & .\\
      8: & . & 3 & 6 & . & .\\
      9: & . & . & . & . & .\\
      10: & . & . & . & . & .\\
      11: & . & \textcolor{red}{3} & 6 & . & .\\
      12: & . & . & . & . & .\\
      13: & . & . & 6 & 9 & .\\
      14: & . & . & . & . & .\\
      15: & . & . & . & . & .\\
      16: & . & . & \textcolor{red}{2} & 3 & .\\
      17: & . & . & . & . & .\\
      18: & . & . & . & \textcolor{red}{3} & 4\\
      19: & . & . & . & . & .\\
      20: & . & . & . & . & \textcolor{red}{-}
      \end{matrix}
      \]
\end{minipage}
\hfill  
\begin{minipage}{.3\textwidth}
\[
\begin{matrix}
       & 0 & 1 & 2 & 3 & 4\\
      0: & 1 & . & . & . & .\\
      1: & . & . & . & . & .\\
      2: & . & . & . & . & .\\
      3: & . & . & . & . & .\\
      4: & . & . & . & . & .\\
      5: & . & 1 & . & . & .\\
      6: & . & . & . & . & .\\
      7: & . & . & . & . & .\\
      8: & . & 6 & . & . & .\\
      9: & . & . & . & . & .\\
      10: & . & . & 6 & . & .\\
      11: & . & \textcolor{red}{3} & . & . & .\\
      12: & . & . & . & . & .\\
      13: & . & . & 12 & . & .\\
      14: & . & . & . & . & .\\
      15: & . & . & . & 9 & .\\
      16: & . & . & \textcolor{red}{2} & . & .\\
      17: & . & . & . & . & .\\
      18: & . & . & . & \textcolor{red}{6} & .\\
      19: & . & . & . & . & .\\
      20: & . & . & . & . & \textcolor{red}{4}
      \end{matrix}
\]
\end{minipage}
\hfill
\begin{minipage}{.3\textwidth}
\[
\begin{matrix}
       & 0 & 1 & 2 & 3 & 4\\
      0: & 1 & . & . & . & .\\
      1: & . & . & . & . & .\\
      2: & . & . & . & . & .\\
      3: & . & 2 & . & . & .\\
      4: & . & . & . & . & .\\
      5: & . & 1 & 2 & . & .\\
      6: & . & 2 & . & . & .\\
      7: & . & . & . & . & .\\
      8: & . & 4 & 8 & . & .\\
      9: & . & . & . & . & .\\
      10: & . & . & 4 & 6 & .\\
      11: & . & \textcolor{red}{1} & 2 & . & .\\
      12: & . & . & . & . & .\\
      13: & . & . & 4 & 6 & .\\
      14: & . & . & . & . & .\\
      15: & . & . & . & 3 & 4\\
      16: & . & . & \textcolor{red}{-} & . & .\\
      17: & . & . & . & . & .\\
      18: & . & . & . & \textcolor{red}{-} & .\\
      19: & . & . & . & . & .\\
      20: & . & . & . & . & \textcolor{red}{-}
      \end{matrix}
      \]
\end{minipage}
\caption{Betti tables for three scrolls in $\bfP(1^2, 3^2, 6^3)$ with profiles $(1,3,6^3)$ (left), $(3^2, 6^3)$ (center), and $(1,3^2,6^2)$ (right). The red entry in the $i^{\text{th}}$ column shows the last position in the column that may contain a nonzero entry if the $wN_i$ condition is to be satisfied.}
\label[figure]{fig_bettiTables}
\end{figure}

A somewhat orthogonal notion to the linearity of the resolution as measured by the $N_p$ properties is the Castelnuovo-Mumford regularity of the variety. In weighted projective space, there are a couple of different notions of regularity, which coincide when all the weights are one. We can investigate them here using our knowledge of the Betti table of the scroll.

\begin{definition}[{\cite[Section 5]{Ben04}; \cite[Definition 1.5]{Brown_Erman_Positivity}}]
Let $M$ be a module over $S(\bfw)$. The \textbf{(weighted) Castelnuovo-Mumford regularity} of $M$ is the minimal $r$ such that $H^i_\mfm(M)_j = 0$ for all $i\geq 0$ and $j>r-i$. 
The $\textbf{Koszul regularity}$ of $M$ is the minimal $r$ such that $H^i_\mfm(M)_j = 0$ for all $i\geq 0$ and $j>r-w^{i-1}-1$, where $w^i$ is as in \Cref{def_wNk}. 
For a subscheme $X \subseteq \bfP(\bfw)$, we use  $kReg(X)$ and  $wReg(X)$ to denote the Koszul and weighted regularity of $S/I_X$, respectively.
\end{definition}
We immediately see that, for a Cohen--Macaulay subscheme of dimension $d$, the Koszul regularity and weighted regularity differ by the constant $w^d-d$, as the only nonzero local cohomology module is $H^{d+1}_{\mfm}(-)$. For scrolls, we compute the regularity directly from the Betti table.

\begin{cor}\label[corollary]{cor_regularity}
Let $\bfP(\bfw)$ be a divisible weighted projective space and $X$ 
a $d$-dimensional determinantal scroll with profile $(1^{r_0}, m_1^{r_1}, \ldots, m_k^{r_k})$. Then
\[
kReg(X) =  w^d +1 - \sum_{i=0}^k (a_i-r_i)m_i  
\quad \text{and}\quad
wReg(X) =  d + 1 - \sum_{i=0}^k (a_i-r_i)m_i.
\]
\end{cor}

\begin{proof}
By \cite{Symonds2011}, the weighted regularity may be computed as 
$ \text{(height of the Betti table) } - \sum_{i=0}^k a_i(m_i-1).$ 
Since the height of the Betti table is realized by the last column and the projective dimension of $X$ is equal to its codimension $\left(\sum_{i=0}^k r_i\right) -1$, the height of the Betti table is 
$
\max_p\{\tau_p-p\} = 1 + \sum_{i=0}^k r_im_i - \sum_{i=0}^k r_i.
$
This gives the desired weighted regularity. As noted above, the Koszul regularity of a $d$-dimensional Cohen--Macaulay scheme in $\bfP(\bfw)$ differs from the weighted regularity by $w^d-d$.
\end{proof}

Applying these formulas gives us two corollaries, one of which suggests a close relationship between the regularity and the degree, while the other highlights the subtle differences between them. 

\begin{cor}\label[cor]{minDeg_minReg}
Let $\bfP(\bfw)$ be a divisible weighted projective space and $X$ a $d$-dimensional determinantal scroll of minimal degree. Then $X$ has minimal regularity among all $d$-dimensional weighted determinantal scrolls in $\bfP(\bfw)$.
\end{cor}
\begin{proof}
Note first that to minimize the regularity we need to minimize the sum $\sum_{i=0}^kr_im_i$ subject to the constraint $\sum_{i=0}^k r_i = \sum_{i=0}^ka_i - d$. We begin with the former. Let $r = (r_0, \ldots, r_k)$ and consider $r' = (r_0, \ldots, r_{j-1}, r_j+1, r_{j+1}-1, r_{j+2}, \ldots, r_k)$. Then 
\[
\sum_{i=0}^k r'_im_i = \left(\sum_{i=0}^k r_im_i\right) +m_j -m_{j+1}  < \sum_{i=0}^k r_im_i
\]
so we can minimize greedily just as in \cref{noncone_degree}, obtaining the minimal degree profile. \end{proof}

Apart from when all three are minimized, the relationship between the degree and regularity becomes murkier. In particular, we have the following. 

\begin{cor}\label[cor]{cor_kReg_implication}
Let $X$ and $Y$ be two $d$-dimensional weighted determinantal scrolls in the same divisible weighted projective space. Then there are no implications between the equalities \begin{enumerate}
\item $reg(X) = reg(Y)$
\item $\deg(X) = \deg(Y)$
\end{enumerate}
\end{cor}

\begin{proof}
The examples with profiles $(1,6^3)$ and $(3^2, 6^2)$ in \Cref{fig_table} demonstrate that $(2)$ does not imply $(1)$. Meanwhile \Cref{ex_same_deg_dif_wReg} demonstrates that $(1)$ also fails to imply $(2)$.
\end{proof}

\begin{example}\label[example]{ex_same_deg_dif_wReg}
Let $X$ and $Y$ be weighted determinantal scrolls in $\bfP(1^4, 2^4, 4^4),$ with profiles  $(1, 2^4, 4)$ and $(1^3, 2, 4^2)$ respectively. Then $kReg(X) =  kReg(Y)= 6$ and $wReg(X) = wReg(Y) = -8$. On the other hand, $\deg(X) = \frac{13}{4^4}$ while $\deg(Y) = \frac{1}{2^5}$.
\end{example}

We emphasize that \cref{minDeg_minReg} does not hold for minimal degree varieties that are not Cohen-Maculay. In these cases, the behavior of weighted and Koszul regularity diverges. We see in examples that at least the weighted regularity seems to become less coupled with the degree even for minimal degree varieties, as demonstrated in the following example.

\begin{example} \label{three_curves_2}
Returning once again to \Cref{three_curves_1}, we see that  $C_1$ and $C_2$ are Cohen--Macaulay with degree 2, Koszul regularity 2 and weighted regularity 1. On the other hand, a direct local cohomology computation shows that $C_0$, which is not Cohen--Macaulay, has Koszul regularity 2 and weighted regularity 2. In particular, we have a curve of minimal degree that does \emph{not} exhibit minimal weighted regularity.
\end{example}

Notice that if we take an $m_i$-cone over $X$, this corresponds to incrementing $a_i$ by 1 while the Betti table remains unchanged, so the weighted regularity decreases by $m_i-1$. On the other hand, if $m_i\geq w_{r}$, then the Koszul regularity of the cone will remain unchanged, while if $m_i< w_r$ the Koszul regularity will increase by $w_r-m_i$. In particular, the weighted regularity is unchanged by the addition of degree 1 variables, while the Koszul regularity is unchanged by the addition of variables of degree at least the maximum degree of a current variable. Comparing with \cref{rmk_minDegCone}, we draw the reader's attention to the similarity of the behavior under the $m$-cone operation of the Koszul regularity and the degree.

\begin{example}\label{ex_cones_reg}
Let $Y_1,Y_2 \subseteq \bfP(1^2, 3^2, 6^3)$ be the determinantal scrolls defined by
\[
I_2\begin{pmatrix}
x_1 & y_1 & y_2 & z_1 & z_2 \\
x_2 & y_2 & x_2^3 & z_2 & y_2^2
\end{pmatrix}
\quad\text{and}\quad
I_2\begin{pmatrix}
x_1 & y_1 & z_1 & z_2 & z_3\\
x_2 & x_1^3 & z_2 & z_3 &y_1^2
\end{pmatrix}.
\]
Both are cones over curves that are minimal degree in their weighted linear span: that is $Y_1$ is a $6$-cone over a minimal degree curve $C_1\subset \bfP(1^2, 3^2, 6^2)$ and $Y_2$ is a $3$-cone over a minimal degree curve $C_2\subset \bfP(1^2, 3, 6^3)$. Both $C_1$ and $C_2$ have weighted regularity 1 and Koszul regularity 6. The cone $Y_1$ also has Koszul regularity 6 and has weighted regularity -4, while  $Y_2$ has Koszul regularity 9 and weighted regularity -1.

On the other hand, if we take a $3$-cone over $C_1$ we get a surface in $\bfP(1^2, 3^3, 6^2)$ with Koszul regularity $9$ and weighted regularity -1, while taking a $1$-cone over $C_1$ gives a surface in $\bfP(1^3, 3^2, 6^2)$ with Koszul regularity $11$ and weighted regularity $1$. In particular, if we consider a fixed ideal and change the ambient ring, the weighted regularity is more stable upon increasing the number of low degree variables, while the Koszul regularity is more stable upon increasing the number of high degree variables.
\end{example}

The degrees and regularities of all weighted determinantal scrolls (including cones) with codimension at least 2 in $\bfP(1^2, 3^2, 6^3)$ are recorded in \Cref{fig_table}.

\begin{figure}[h!]
\small
\[
\renewcommand{\arraystretch}{1.5}
\begin{array}[t]{c c}
\begin{array}[t]{c|c|c|c|c}
dim & profile & deg & kReg & wReg \\
\hline
 1 & \color{red}(1,\ 3^2,\ 6^3) & \color{red}\frac{13}{6} & \color{red}6 & \color{red}1\\
 \hline
 2 & \color{red}(1,\ 3^2,\ 6^2) & \color{red}\frac{1}{3} & \color{red} 6 &\color{red} -4\\
    & (1,\ 3,\ 6^3) & \frac{11}{18} & 9 & -1\\
    & (3^2,\ 6^3) & \frac{7}{6} & 11 & 1\\
\hline
3 & \color{red}(1,\ 3^2,\ 6) & \color{red}\frac{11}{6^3} & \color{red} 6 & \color{red}-9\\
   & (1,\ 3,\ 6^2) & \frac{5}{54} & 9 & -6 \\
   &(1,\ 6^3) & \frac{1}{6} & 12 & -3\\
   & (3^2,\ 6^2) & \frac{1}{6} & 11 & -4\\
   & (3,\ 6^3) & \frac{5}{18} & 14 & -1
\end{array} &
\begin{array}[t]{c|c|c|c|c}
dim & profile & deg & kReg & wReg \\
\hline
4 & \color{red}(1,\ 3\ ,3) & \color{red}\frac{5}{3\cdot 6^3} & \color{red}3 & \color{red}-14\\
   & (1,\ 3\ ,6) & \frac{1}{2\cdot6^2} & 6 & -11\\
   & (1,\ 6\ ,6) & \frac{5}{6\cdot3^3} & 9 & -8\\
   & (3,\ 3\ ,6) & \frac{5}{6^3} & 8 & -9\\
   & (3,\ 6,\ 6) & \frac{1}{3^3} & 11 & -6\\
   & (6,\ 6\ ,6) & \frac{1}{2\cdot3^2} & 14 & -3
\end{array}
\end{array}
\]
\caption{Degree and regularities for weighted determinantal scrolls in $\bfP(1^2, 3^2, 6^3)$. Lines in red correspond to varieties of minimal degree of each dimension.}\label[figure]{fig_table}
\end{figure}

\section{Curves in weighted projective threefolds} \label{section_threefold}

In this section, we study curves in more general weighted projective threefolds. We take the results from the divisible case as a starting point and use them to bound the degree in the nondivisible setting, illustrating how an inductive approach can be effective in obtaining more general bounds. 

Suppose $C\subset \bfP(w_0,w_1,w_2,w_3)$ is an integral Cohen--Macaulay curve defined by the $2\times 2$ minors of a $2\times 3$  matrix over $S(\bfw)$. We will call such $C$ a \textbf{determinantal curve}, and we begin by computing its degree. Note that in more general weighted projective spaces, the existence of 1-generic matrices is not guaranteed. We are thus forced to deal with the more general pseudo 1-generic matrices. As a reminder, the entries of these matrices may have degrees that are not among the weights of the ambient weighted projective space, and the degrees of entries in a given column may not all be the same.

\begin{prop} \label[prop]{det_curve_degree}
Let $M$ be a $2\times 3$ matrix over $S(w_0,w_1,w_2,w_3)$ with profile $(a_1,a_2,a_3; b)$. If the curve $C$ defined by the maximal minors of $M$ is Cohen--Macaulay, then $C$ has degree
\[
\frac{b(a_1+a_2+a_3+b) + a_1a_2 + a_1a_3 + a_2a_3}{w_0w_1w_2w_3}.
\]
\end{prop}

\begin{proof}
The homogeneous coordinate ring of $C$ is resolved over $S(\bfw)$ by the Eagon-Northcott complex on $M$, which has the form
$$0 \to S(-a_1-a_2-a_3-2b) \oplus S(-a_1-a_2-a_3-b)
\to \bigoplus_{1 \le i < j \le 3} S(-a_i-a_j-b)
\to S \to 0.
$$
From this we obtain the Hilbert series
\[
HS_C(t) = \frac{1-t^{a_1+a_2+b}-t^{a_1+a_3+b}-t^{a_2+a_3+b}+t^{a_1+a_2+a_3+b}+t^{a_1+a_2+a_3+2b}}{(1-t^{w_0})(1-t^{w_1})(1-t^{w_2})(1-t^{w_3})}
\]
Using polynomial long division to divide the numerator by $(1-t)^2$, we can reduce this to
\[
\frac{P_C(t)}{(1+t+\cdots+t^{w_0-1})(1+t+\cdots + t^{w_1-1})(1-t)^{w_2}(1-t)^{w_3}}
\]
where $P_C(t)$ is a polynomial of degree $a_1+a_2+a_3+2b-2$ with coefficients
\begin{multline}
1,2,\ldots, b, b+2, b+4, \ldots, b+2a_1, b+2a_1+1, b+2a_1+2, \ldots, b+a_1+a_2, \\ \underbrace{b+a_1+a_2, \ldots, b+a_1+a_2}_{a_3-a_2}, b+a_1+a_2-1, b+a_1+a_2-2, \ldots, 1.
\end{multline}
By \cref{thm_degree}, the degree of $C$ is
\begin{align*}
\deg(C) &= \frac{P_C(1)}{w_0w_1w_2w_3} \\
&= \frac{\left(\displaystyle\sum_{i=1}^b i\right) + \left( \displaystyle\sum_{i=1}^{a_1}b+2i\right) + \left(\displaystyle\sum_{i=1}^{a_2-a_1}b+2a_1+i \right) + (a_3-a_2)(b+a_1+a_2) + \left(\displaystyle\sum_{i=1}^{b+a_1+a_2-1}i \right)}{w_0w_1w_2w_3}
\end{align*}
which simplifies to the expression in the statement.
\end{proof}

 \begin{example} \label[example]{candidate_11mn} Consider $\bfP(1,1,m,n)$ with coordinates $x_1,x_2,y,z$, in order of increasing degree. Let $k = \floor{\frac{n}{m}}$ and consider the matrix
$$
M = \begin{pmatrix}
 x_1 & x_2^{(k+1)m-n} & y^k   \\
 x_2^{n-km+1} & y  & z 
\end{pmatrix}.
$$
Note that $M$ is pseudo 1-generic and $X = V(I_2(M))\subseteq \bfP(1,1,m,n)$ is an integral curve. Since $M$ has profile $(1,(k+1)m-n,mk;n-km)$, \Cref{det_curve_degree} shows that the degree of $X$ is
 $1 + \frac{1}{n} + \frac{k}{n}$. 
 \end{example}
\begin{prop} \label[prop]{threefold_bound_weak} Let $C$ be a non-degenerate integral curve on $\bfP(1,1,m,n)$. Then 
$$
\deg(C)  
\geq 1 + \frac{1}{m} - \frac{n\bmod m}{mn}. 
$$
\end{prop}
\begin{proof} 
Let $C$ be a curve of degree less than $1 + \frac{1}{m} - \frac{n\bmod m}{mn} $. Take all of the degree $m$ monomials in $x_1, x_2, y$ along with $z$. This gives a morphism 
$$
\psi:\bfP(1,1,m,n) \to \bfP(1^{m+2},n).
$$
Note that regardless of whether this map is generically injective we have $\deg \psi(C) \leq m \deg(C)$.
Thus, the image $\psi(C)$ has degree less than $m +1  - \frac{n\bmod m}{n} $. Since a curve of minimal degree in $\bfP(1^{m+2},n)$ has degree $m+1 + \frac{1}{n}$, it follows that $\psi(C)$ lies on a hyperplane. If the hyperplane contained $z$, then $\psi(C)$ would lie in $\bfP(1^{m+2})$ and, since the minimal degree in this projective space is $m+1$, it follows that $\psi(C)$ is contained in another hyperplane not involving $z$. Regardless,  $\psi(C)$ is contained in a hyperplane not involving $z$. If it involves $y$, then the pullback also involves $y$ and $C$ would be degenerate. If it only involved the weight 1 variables, then the pullback of the equation is a homogeneous polynomial in two variables. This completely factors and once again $C$ is degenerate, a contradiction.
\end{proof}

As can be gleaned from the proof, the bound might not be sharp when the map $\psi$ is not generically injective. We conjecture that the determinantal curve constructed in \Cref{candidate_11mn} has minimal degree.

\begin{conjecture} \label[conjecture]{threefold_conjecture} Let $C$ be a non-degenerate integral curve on $\bfP(1,1,m,n)$. Then 
$$
\deg(C)  
\geq  1 + \frac{1}{m} + \frac{1}{n} - \frac{n \bmod m}{mn}.
$$
\end{conjecture}

\begin{remark} Using \Cref{prop_degree_bound} and \Cref{threefold_bound_weak}, we see that \Cref{threefold_conjecture} is true when  $n \equiv 0$ or $m-1 (\bmod m)$.
\end{remark}

Finally, we end with an example of a rational curve in a three-dimensional weighted projective space whose degree is less than that of any determinantal curve. We consider the weighted projective space $\bfP(1,3,4,7)$ with coordinates $x,y,z,w$, in order of increasing degree. Consider the map $\bfP^1\to\bfP(1,3,4,7)$ given by the weighted series 
$$
\{s^{12}, s^{15}t^{21}, s^6 t^{42}, t^{84}\}\subset \bigoplus_{i\in\Z}H^0(\bfP^1, \mcO(12)^{\otimes i}).
$$
The image of this map is a rational curve of degree $\frac{12}{21} = \frac{4}{7}$. This curve is in fact a complete intersection, cut out by the ideal $(y^2-x^2z, z^2-xw)$. 

\begin{prop} \label{prop_1347}Every integral determinantal curve $C$ in $\bfP(1,3,4,7)$ has degree strictly larger than $\frac{4}{7}$.
\end{prop}
\begin{proof}
Let $M$ be the $2\times 3$ matrix whose minors define $C$.
By \Cref{prime_is_generic}, $M$ must be pseudo 1-generic. Let $M$ have profile $(a_1,a_2,a_3;b)$, so that the degree of $C$ is as in \Cref{det_curve_degree}. We may assume that $a_1\leq a_2\leq a_3$, so the degrees of the entries of $M$ look like 
 \[
 \begin{pmatrix}
 a_1 & a_2 & a_3\\
 a_1+b & a_2+b & a_3+b
 \end{pmatrix}.
 \]
 We note some immediate restrictions on the profile of $M$:
\begin{enumerate}
\item Since there is only one variable, $x$, of degree 1 and no variable of degree 2, $a_2$ must be at least 3; otherwise, the first two entries of the top row of $M$ would share the common factor $x$, violating pseudo 1-genericity. Similarly, $a_3 \geq 4$.
\item If $a_1 = 1$, then $b \geq 2$; otherwise, both entries in the first column would share the common factor $x$, again violating pseudo 1-genericity. Similarly, if $a_1 = 2$, then $b \geq 1$.
\item Since $C$ is non-degenerate, each of the variables $x,y,z,w$ in $S(\bfw)$ must appear somewhere in $M$, so we must have that $a_3+b\geq 7$.
 \end{enumerate}
By \Cref{det_curve_degree},  if $\deg(C) \leq\frac{4}{7}$, the profile $(a_1,a_2,a_3;b)$ must satisfy
 \begin{equation}\label[equation]{eq_48}
 b(a_1+a_2+a_3+b) + a_1a_2 + a_1a_3 + a_2a_3 \leq 48.
 \end{equation}
Given the earlier constraints, we obtain $b(10+a_1) + 7a_1+ 12  \leq 48$. The only options are $a_1 = 1$ and $b = 2$ or $a_1 = 2$ and $b=1$ (see point (2) above). The first case forces $a_3 \geq 5$ which in turn gives us the profile $(1,3,5;2)$ and the second case forces $a_3 \geq 6$ which gives the profile $(2,3,6;1)$.
If $M$ has profile $(2,3,6;1)$, then the top left entry must be $x^2$ (up to scalar multiple). We can then use row and column operations to clear all multiples of $x^2$ out of the first row and column so that $M$ has the form 
\[
\begin{pmatrix}
x^2 & c_1y & c_2y^2\\
\star & \star & \star
\end{pmatrix}
\]
for some $c_i \in \kk$.
This is clearly not pseudo 1-generic.
If $M$ has profile $(1,3,5;2)$, then we can assume the top left entry is $x$. Once again, row and column operations allow us to put $M$ into the form 
\[
\begin{pmatrix}
x & \star & 0\\
\star & \star& \star
\end{pmatrix}
\]
This is also not pseudo 1-generic. Thus, $\deg(C) > \frac{4}{7}$, as required.
\end{proof} 
 
\begin{remark} There are integral determinantal curves in $\bfP(1,3,4,7)$. For example, the scheme defined by
$$
I_2
\begin{pmatrix}
x & y & z \\
x^4+z & x^6+x^3y & w
\end{pmatrix}
$$
is an integral curve of degree $\frac{13}{14}$.
\end{remark}

In addition to \Cref{threefold_conjecture}, we conclude with several natural questions regarding \textit{divisible weighted projective spaces}. We have seen that not all varieties of minimal degree are arithmetically Cohen--Macaulay; as a starting point, it would be useful to characterize this class.

\begin{question} 
Are all arithmetically Cohen--Macaulay varieties of minimal degree in $\bfP(\bfw)$ determinantal?
\end{question}

We know that all curves of minimal degree are rational and smooth.

\begin{question}
Determine which maps from stacky projective lines to $\bfP(\bfw)$ correspond to curves of minimal degree in $\bfP(\bfw)$.
Which have images that are arithmetically Cohen--Macaulay?
\end{question}

\begin{question} 
Geometrically classify the curves of minimal degree in $\bfP(\bfw)$. Let $Q(t)$ be the Hilbert quasi-polynomial of a curve of minimal degree in $\bfP(\bfw)$. To answer this question, one can try to understand the components of the Hilbert scheme $\Hilb^{Q(t)}(\bfP(\bfw))$. Is it irreducible? If not, what are its components, and what do their general members look like? Are the components generically smooth? A good starting case would be the curves in $\bfP(1,1,2,2)$. 
\end{question}

We showed in \Cref{minDeg_minReg} that scrolls of minimal degree have minimal Koszul and weighted regularity, but saw in \Cref{three_curves_2} that in the non Cohen--Macaulay case the weighted regularity at least may increase. How general is this phenomenon?

\begin{question} 
Let $X, Y \subseteq \bfP(\bfw)$ be two varieties of minimal degree, and assume that $X$ is arithmetically Cohen--Macaulay and $Y$ is not. How do the weighted and Koszul regularities of $X$ and $Y$ compare? 
\end{question}

As seen in \Cref{not_proj_equiv}, the Kronecker--Weierstrass form is not determined up to projective equivalence.

\begin{question} 
Describe the $\text{Aut}(\bfP(\bfw))$-orbits of the weighted $1$-generic matrices; that is, determine when two Kronecker--Weierstrass forms are projectively equivalent. As a first step, describe the $GL$-orbits of $2\times n$ matrices of linear forms.
\end{question}

\bibliography{references.bib}

\end{document}